\documentclass[11pt, a4paper]{amsart}
\usepackage[top=3cm,bottom=2.6cm, outer=3cm, inner=3cm]{geometry}
\usepackage{amsfonts,amsmath,amssymb,bbm,amsthm}
\usepackage{enumitem,lipsum,graphicx,csquotes}
\usepackage{stmaryrd}
\usepackage{setspace,scrpage2}

\usepackage{hyperref}

\begin{document}
   
\newcommand{\R}{\mathbb R}
\newcommand{\Z}{\mathbb Z}
\newcommand{\E}{\mathbb E}
\newcommand{\1}{\mathbbm 1}
\newcommand{\N }{\mathbb N}
\newcommand{\Q }{\mathbb Q}
\newcommand{\q }{\mathsf Q}
\newcommand{\p }{\mathsf P}
\newcommand{\F }{\mathcal F}
\newcommand{\G }{\mathcal G}
\newcommand{\D }{\mathcal D}
\newcommand{\h }{\mathcal H}
\newcommand{\M }{\mathcal M}
\newcommand{\eps}{\varepsilon}
\newcommand{\argmax}{\operatornamewithlimits{argmax}}
\newcommand{\com}[1]{{\color{blue} \textbf{[#1]}}}

\newcommand{\nc}{\newcommand}
\nc{\bg}{\begin} \nc{\e}{\end} \nc{\bi}{\begin{itemize}} \nc{\ei}{\end{itemize}} \nc{\be}{\begin{enumerate}} \nc{\ee}{\end{enumerate}} \nc{\bc}{\begin{center}} \nc{\ec}{\end{center}} \nc{\bq}{\begin{equation}} \nc{\eq}{\end{equation}} \nc{\la}{\label} \nc{\fn}{\footnote} \nc{\bs}{\backslash} \nc{\ul}{\underline} \nc{\ol}{\overline} \nc{\np}{\newpage} \nc{\btab}{\begin{tabular}} \nc{\etab}{\end{tabular}} \nc{\mc}{\multicolumn} \nc{\mr}{\multirow} \nc{\cdp}{\cleardoublepage} \nc{\fns}{\footnotesize} \nc{\scs}{\scriptsize} \nc{\RA}{\Rightarrow} \nc{\ra}{\rightarrow} \nc{\rc}{\renewcommand} \nc{\bfig}{\begin{figure}} \nc{\efig}{\end{figure}} \nc{\can}{\citeasnoun} \nc{\vp}{\vspace} \nc{\hp}{\hspace}\nc{\LRA}{\Leftrightarrow}\nc{\LA}{\Leftarrow}\nc{\rr}{\q}

\renewcommand{\tilde}{\widetilde}

\nc{\ch}{\chapter}
\nc{\s}{\section}
\nc{\subs}{\subsection}
\nc{\subss}{\subsubsection}

\newtheorem{thm}{Theorem}[section]
\newtheorem{prop}[thm]{Proposition}
\newtheorem{cor}[thm]{Corollary}
\newtheorem{lem}[thm]{Lemma}

\theoremstyle{remark}
\newtheorem{rem}[thm]{Remark}

\theoremstyle{remark}
\newtheorem{ex}[thm]{Example}

\theoremstyle{remark}
\newtheorem{cex}[thm]{Counterexample}

\theoremstyle{definition}
\newtheorem{df}[thm]{Definition}

\newenvironment{rcases}{
  \left.\renewcommand*\lbrace.
  \begin{cases}}
{\end{cases}\right\rbrace}

\setlength\parindent{0pt}
\pagestyle{plain}

 \author{D\"orte Kreher}
   \address{Humboldt-Universit\"at zu Berlin, Germany}
   \email{kreher@math.hu-berlin.de}


  \thanks{Financial support from {\it SNF grant 137652} and {\it CRC 649: Economic Risk} is gratefully acknowledged. The author thanks Ashkan Nikeghbali, Monique Jeanblanc, Philip Protter, and an anonymous referee for their careful reading of the manuscript and for their helpful comments and remarks.} 

\title{Change of measure up to a random time: Details}
\date{\today}

\maketitle

\begin{abstract}
This paper extends results of Mortimer and Williams (1991) about changes of probability measure up to a random time under the assumptions that all martingales are continuous and that the random time avoids stopping times.
We consider locally absolutely continuous measure changes up to a random time, changes of probability measure up to and after an honest time, and changes of probability measure up to a pseudo-stopping time. Moreover, we apply our results to construct a change of probability measure that is equivalent to the enlargement formula and to build for a certain class of pseudo-stopping times a class of measure changes that preserve the pseudo-stopping time property.
Furthermore, we study for a price process modeled by a continuous semimartingale the stability of the No Free Lunch with Vanishing Risk (NFLVR) property up to a random time, that avoids stopping times, in the progressively enlarged filtration and provide sufficient conditions for this stability in terms of the Az\'ema supermartingale. 
\end{abstract}

\s{Introduction}

Motivated by models from physics and chemistry Mortimer and Williams (1991) study how to perform a change of measure up to a random time $\sigma$ on a filtered probability space $(\Omega,\F,(\F_t),\p)$. More precisely, in their paper titled "Change of measure up to a random time: Theory" they derive the semimartingale decomposition of continuous $(\p,\F_t)$-martingales up to time $\sigma$ in the progressively enlarged filtration
\[\G'_t=\F_t\vee\sigma\left(\1_{\{\sigma>s\}};s\leq t\right)\]
under an equivalent probability measure $\q$ and they give the expression of the $(\q,\G'_t)$-hazard function of $\sigma$. To prove their result they use elementary methods and do no rely on the theory of enlargement of filtrations. Besides, Mortimer and Williams (1991) claim in their paper that \enquote{it is the \textit{examples} which make this topic of some interest}, but the only two examples they provide deal with the well-known path decomposition of the standard Brownian motion.

In this paper we extend their result in several ways and provide interesting classes of examples working under the standing assumptions that $\sigma$ avoids stopping times and that all $(\F_t)$-martingales are continuous. As in Mortimer and Williams (1991)\nocite{Mortimer} we do no rely on any deep results from the theory of enlargements of filtrations, but choose a rather direct approach using only elementary methods to prove our results. 

First, we extend their result to locally absolutely continuous changes of measure up to a random time, which allows us to construct a  change of probability measure that is equivalent to the enlargement formula up to time $\sigma$. Second, we study changes of probability measure for honest times. Honest times are known to be well-suited for a progressive enlargement of filtration since the seminal work of Barlow (1978)\nocite{Barlow}, because in this case all $(\F_t)$-semimartingales remain semimartingales in the enlarged filtration on the whole time horizon. Therefore, if $\sigma$ is an honest time, we are able to extend the Girsanov-type theorem from Mortimer and Williams (1991) after time $\sigma$. While the result itself is not very surprising, the way we prove it is interesting because as in Mortimer and Williams (1991) we do not assume any prior knowledge of the theory of enlargements of filtrations. Actually, as it turns out there is a nice link to so called relative martingales, which were studied by Az\'ema, Meyer, and Yor (1992). 
Third, we study changes of measure up to pseudo-stopping times which were introduced by Nikeghbali and Yor (2005). 
As finite-valued honest times are ends of optional sets, their definition is independent of the underlying probability measure. This is however not true for pseudo-stopping times. An interesting and challenging problem is therefore to analyze the stability of the pseudo-time property under different probability measures. For a special class of pseudo-stopping times we are able to provide a class of equivalent probability measures which preserve the pseudo-stopping time property. 
Finally, we also generalize the example of the Brownian path decomposition given in Mortimer and Williams (1991). As opposed to Mortimer and Williams (1991), who provide a Markovian study of this example, our analysis is based on semimartingale calculus only and uses the specific structure of the Az\'ema supermartingale of a class of pseudo-stopping times associated with honest times. 

The last part of the paper deals with the question of no arbitrage up to a random time. Since the technique of progressively enlarging a filtration with a random time is a standard tool in mathematical finance to model credit risk and insider trading, this question is of particular interest. In the recent literature it is addressed in a couple of papers under different assumptions on the price process, the random time, and the precise no arbitrage concept, cf.~Fontana, Jeanblanc, and Song (2013); Acciaio, Fontana, and Kardaras (2016); Aksamit, Choulli, Deng, and Jeanblanc (2014). To be concrete, we deal with the following question assuming that the stock price process is modeled by a continous $(\F_t)$-semimartingale: if the market satisfies the condition of ``no free lunch with vanishing risk'' (NFLVR) with respect to the filtration $(\F_t)$, under which conditions does the market then also satisfy NFLVR with respect to the progressively enlarged filtration until time $\sigma$? For example, it is known that honest times allow for arbitrage on the time horizon $[0,\sigma]$ in the progressively enlarged filtration, cf.~Fontana, Jeanblanc, and Song (2013). In this paper we consider an arbitrary random time $\sigma$ that avoids $(\F_t)$-stopping times and derive sufficient criteria for the validity of NFLVR up to time $\sigma$, assuming that all $(\F_t)$-martingales are continuous. To do this we choose two different approaches: in the first one we work directly with the definition of NFLVR, while in the second one we identify a local martingale deflator in the enlarged filtration and check under what conditions it is a uniformly integrable martingale, which would imply that NFVLR is satisfied up to time $\sigma$ due to the fundamental theorem of asset pricing. 

This paper is organized as follows: In the next section after having introduced the general setup and notation, we recall the result from \cite{Mortimer} and give some first corollaries thereof. In Subsection \ref{mult} we moreover introduce the multiplicative decomposition of the Az\'ema supermartingale, which will be used frequently during the paper. Section \ref{abs} deals with locally absolutely continuous measure changes. Afterwards we specialize our analysis to the class of honest times in Section \ref{hon} and to the class of pseudo-stopping times in Section \ref{pst}. Section \ref{generalize} contains a generalization of the example given in \cite{Mortimer}. Finally, in Section \ref{finance} we consider the question of no arbitrage up to a random time and provide sufficient criteria for the validity of NFLVR in the progressively enlarged filtration up to time $\sigma$.

\s{General theory}\label{general}

\subs{Setup and notation}\label{notations}

Throughout the paper we work on a filtered probability space $(\Omega,\F,(\F_t)_{t\geq0},\p)$, where 
$(\F_t)$ is assumed to satisfy the natural conditions, i.e.~$(\F_t)$ is right-continuous and $\F_0$ contains all $\sigma$-negligible sets. Here, a subset $A\subset\Omega$ is called $\sigma$-negligible if there exists a sequence $(B_n)_{n\geq0}$ of subsets of $\Omega$ such that $A\subset\bigcup_{n\in\N}B_n$ and such that for all $n\in\N$, $B_n\in\F_n$ with $\p(B_n)=0$. It was shown in \cite{newkind} that under the natural conditions any martingale admits a c\`adl\`ag modification and we will always work with this modification in the following. If $(X_t)$ is a real-valued stochastic process we denote by 
\[\ol{X}_t:=\sup_{s\leq t}X_s\qquad\text{ and }\qquad\ul{X}_t:=\inf_{s\leq t}X_s,\qquad t\geq0, \]
its supremum resp.~infimum process and by $T^X_a=\inf\{t>0:\ X_t=a\}$ the first hitting time of the level $a\in\R$. Note that under the natural assumptions $T^X_a$ is a stopping time, if $X$ is a right-continuous adapted process, cf.~\cite{newkind}. Furthermore, $\mathcal{M}(\p,\F_t)$ denotes the set of $(\p,\F_t)$-martingales and $\mathcal{M}_{loc}(\p,\F_t)$ resp.~$\mathcal{M}_{u.i.}(\p,\F_t)$ the set of local resp.~uniformly integrable $(\p,\F_t)$-martingales.

Finally, we denote by $\sigma:\ \Omega\ra[0,\infty]$ we denote an $\F$-measurable random time, which gives rise to the progressively enlarged filtration
\[\G_t:=\bigcap_{s>t}\left(\F_s\vee\sigma(\1_{\{\sigma>r\}};r\leq s)\right).\]

\textbf{Throughout the paper we will assume that the following two assumptions are satisfied:}\\

{\bf (A)} $\sigma$ avoids any $(\F_t)$-stopping time: $\p(\sigma=T)=0$ for any $(\F_t)$-stopping time $T$.\\ 
{\bf (C)} All $(\F_t)$-martingales are continuous.\\

We denote by $Z^\p_t:=\p(\sigma>t|\F_t)$ the Az\'ema supermartingale of $\sigma$. It decomposes as $Z_t^\p=m_t^\p-A_t^\p$ with $m_t^\p=\E^\p(A^\p_\infty|\F_t)$ being a uniformly integrable martingale and $(A_t^\p)$ being the $(\F_t)$-dual optional projection of the process $(\1_{\{\sigma\leq t\}})_{t\geq0}$. We note that under assumption (C), $A^\p$ is also the dual predictable projection of $(\1_{\{\sigma\leq t\}})_{t\geq0}$, since in this case the predictable and optional sigma fields related to $(\F_t)$ coincide. Moreover, note that under the assumptions (AC) the Az\'ema supermartingale is continuous and $Z_t^\p=m_t^\p-A_t^\p$ is thus its Doob-Meyer decomposition. For the definitions and properties of dual optional and dual predictable projections we refer the reader to chapter VI.2 of \cite{DM2}.\\

Let $\rho$ be a non-negative $\F$-measurable random variable with expectation one. Then $\q:=\rho.\p$ defines a new probability measure which is absolutely continuous to $\p$. We denote by $(\rho_t)$ resp.~$(\tilde{\rho}_t)$ the optional projection of $\rho$ on $(\F_t)$ resp.~$(\G_t)$ satisfying for all $t\geq0$,
\[\rho_t:=\E^\p(\rho|\F_t),\quad \tilde{\rho}_t:=\E^\p(\rho|\G_t),\]
where $(\tilde{\rho}_t)$ is chosen to be c\`adl\`ag and $(\rho_t)$ is continuous due to (C). Furthermore, we define the $(\p,\F_t)$-supermartingale
\[h_t:=\E^\p(\rho\1_{\{\sigma>t\}}|\F_t),\quad t\geq0.\]
By Bayes' formula one has
\[h_t=\rho_t\cdot\q(\sigma>t|\F_t)=:\rho_tZ_t^\q.\]
Since $\sigma$ avoids stopping times, $\p(\sigma=\infty)=0$ and $\sigma$ is finite $\p$-almost surely. Therefore, $Z^\p$ and $h$ both converge towards zero almost surely as $t\ra\infty$.\\
If $h$ is strictly positive, we denote by $\mu$ the stochastic logarithm of $h$, i.e.~$h_t=\mathcal{E}(\mu)_t$. The process $\mu$ is again a $(\p,\F_t)$-supermartingale with Doob-Meyer decomposition $\mu=\mu^L-\mu^F$, where $\mu^L\in\mathcal{M}_{loc}(\p,\F_t)$ and $\mu^F$ is increasing. Moreover, $h,\mu,\mu^L$, and $\mu^F$ are all continuous. If $h$ is not strictly positive, then the process $\mu$ and hence also $\mu^L$ and $\mu^F$ are only well-defined on the stochastic interval $[0,T_0^h)$. 

\subs{Girsanov-type theorems}

We are now ready to recall the result of \cite{Mortimer}, Lemma 2.

\begin{thm}\label{original}
Assume that $h$ is strictly positive and let $U=(U_t)_{t\geq0}$ be a local $(\p,\F_t)$-martingale. Then the process $\left(\1_{\{\sigma>t\}}V_t\exp(\mu^F_t)\right)_{t\geq0}$ is a local $(\q,\G_t)$-martingale, where $V:=U-\langle U,\mu\rangle$.\\ Moreover, the process $\left(\mu^F_{t\wedge \sigma}\right)_{t\geq0}$ is the $(\q,\G_t)$-dual predictable projection of $(\1_{\{\sigma\leq t\}})_{t\geq0}$.
\end{thm}

\begin{proof}
The claim is proven in \cite{Mortimer} for $\G_t'=\F_t\vee\sigma\left(\1_{\{\sigma>s\}};s\leq t\right)$ instead of $\G_t$ defined above. However, since $(\G_t')$-martingales remain martingales with respect to the right-continuous augmentation $(\G_t)$ of $(\G_t')$, the claim follows easily.
\end{proof}

As an immediate consequence of the above result we deduce 

\begin{cor}\label{cor1}
Assume that $h$ is strictly positive. If $U\in\mathcal{M}_{loc}(\p,\F_t)$, then 
\[V_{t\wedge\sigma}=U_{t\wedge\sigma}-\langle U,\mu\rangle_{t\wedge\sigma}\in\mathcal{M}_{loc}(\q,\G_t).\]
\end{cor}

\begin{proof}
Taking $U\equiv1$ in Theorem \ref{original} yields that 
\[H_t:=\1_{\{\sigma>t\}}\exp\left(\mu^F_t\right)\in\mathcal{M}_{loc}(\q,\G_t).\]
Since $V$ is continuous, $[H,V]\equiv0$. Hence, the product $HV$ is a local $(\q,\G_t)$-martingale if and only if $V$ is also a local $(\q,\G_t)$-martingale as long as $H_{t-}>0$, i.e.~on the interval $[0,\sigma]$.
\end{proof}

\begin{rem}
If we choose $\rho\equiv1$ in the above corollary, we recover the enlargement formula up to time $\sigma$ (cf.~\cite{DM5}, paragraph XX.76): For any $M\in\mathcal{M}_{loc}(\p,\F_t)$ we have 
\begin{equation}\label{ef}
M_{t\wedge\sigma}-\int_0^{t\wedge\sigma}\frac{d\langle M,Z^\p\rangle_s}{Z^\p_s}=M_{t\wedge\sigma}-\int_0^{t\wedge\sigma}\frac{d\langle M,m^\p\rangle_s}{Z^\p_s}\in\mathcal{M}_{loc}(\p,\G_t).
\end{equation}
\end{rem}

\begin{rem}
In \cite{Mortimer} the authors prove their result without applying any results from the theory of progressive enlargement of filtrations. Of course, Corollary \ref{cor1} can also be proven by applying first Girsanov's theorem and afterwards the enlargement formula under $\q$. For so called honest times this is done in paragraph XX.81 of \cite{DM5}, where a more general version of the above result is proven without relying on the assumptions (AC).
\end{rem}

Next we show that Theorem \ref{original} also holds if $h$ is not necessarily strictly positive.

\begin{thm}\label{h0}
If $U=(U_t)_{t\geq0}$ is a local $(\p,\F_t)$-martingale, then $X_t:=\1_{\{\sigma>t\}}V_t\exp(\mu^F_t)$ and $V_{t\wedge\sigma}$ are local $(\q,\G_t)$-martingales, where $V_{t\wedge\sigma}=U_{t\wedge\sigma}-\langle U,\mu\rangle_{t\wedge\sigma}$.
\end{thm} 

\begin{proof}
First we show that $\q\left(\sigma<T_0^h\right)=1$. For this note that $T_0^h=T_0^\rho\wedge T_0^{Z^\q}$ because $h_t=\rho_tZ_t^\q$. But we have
\[\q\left(T_0^\rho<\infty\right)=\E^\p\left(\rho\1_{\{T_0^\rho<\infty\}}\right)=\E^\p\left(\rho_\infty\1_{\{T_0^\rho<\infty\}}\right)
=\E^\p\left(0\cdot\1_{\{T_0^\rho<\infty\}}\right)=0.\]
Since $\sigma$ avoids stopping times under $\p$ and $\q$ is absolutely continuous to $\p$, $\q(\sigma=T_0^h)=0$ and $\sigma$ is also $\q$-almost surely finite. Hence,
\[\q\left(\sigma\geq T_0^h\right)=\q\left(\sigma>T_0^h\right)=\q\left(\sigma>T_0^{Z^\q}\right)=\E^\q Z^\q_{T_0^{Z^\q}}=0.\]
Especially, this means that $X$ is $\q$-a.s.~well-defined since $\mu$ is well-defined on the interval $[0,T_0^h)$.
Second for every $n\in\N$ we write $U^n_t:=U_{t\wedge T_{1/n}^h},\ t\geq0$.  According to Theorem \ref{original}, the process $X^n_t:=X_{t\wedge T_{1/n}^h}$ is a local $(\q,\G_t)$-martingale for every $n\in\N$. Therefore, $X$ is a local $(\q,\G_t)$-martingale on the interval $\left[0,T_0^h\right)=\bigcup_{n\in\N}\left[0,T_{1/n}^h\right]$
 and since $\left[0,T_0^h\right)\supset [0,\sigma]$ $\q$-almost surely, this implies that 
\[X_t=\1_{\{\sigma>t\}}V_t\exp(\mu^F_t)\in\mathcal{M}_{loc}(\q,\G_t).\]
Finally, $(V_{t\wedge\sigma})$ is a local $(\q,\G_t)$-martingale by the same reasoning as in the proof of Corollary \ref{cor1}.
\end{proof}

\subs{Multiplicative decomposition of the Az\'ema supermartingale}\label{mult}

Before we come to further extensions and applications, we introduce in this subsection the so called It\^o-Watanabe decomposition of the Az\'ema supermartingale, which will be frequently used in the following sections. Since it is less known than the Doob-Meyer decomposition, we briefly recall a continuous version of the result from \cite{ItoWata}, cf.~also \cite{surmartingale}.

\begin{thm}\label{IW}
Let $Z$ be a continuous non-negative supermartingale with Doob-Meyer decomposition $Z=m-A$. Then $Z$ factorises uniquely as $Z=ND$, where $N$ is a continuous non-negative local martingale starting from $N_0=1$ and $D$ is a continuous decreasing process such that both $N$ and $D$ are constant on the set $\{Z=0\}$. Moreover, $N$ and $D$ are given by
\[D_t=Z_0\exp\left(-\int_0^{t\wedge T_0^Z}\frac{dA_s}{Z_s}\right),\quad 
N_t=\mathcal{E}\left(\int_0^{t\wedge T_0^Z}\frac{dm_s}{Z_s}\right).\]
\end{thm}

\begin{rem}\label{nd}
If $Z=Z^\p$ is the Az\'ema supermartingale of $\sigma$, then
\[Z_t^\p=0\quad\LRA\quad m_t^\p-A^\p_t=\E^\p(A^\p_\infty-A^\p_t|\F_t)=0\quad\LRA\quad A_t^\p=A_s^\p\ \forall\ s\geq t,\]
since $A^\p$ is an increasing process. Therefore, $A^\p$ and $m^\p$ only move on the set $\{Z^\p>0\}$. Hence, in this case the processes
\[D_{t}^\p=\exp\left(-\int_0^{t}\frac{dA^\p_s}{Z^\p_s}\right),\quad N_{t}^\p=\mathcal{E}\left(\int_0^{t}\frac{dm^\p_s}{Z^\p_s}\right)\]
are well-defined and fulfill $supp(dD^\p)\subset\{Z^\p>0\}$ resp.~$supp(d\langle N^\p\rangle)\subset\{Z^\p>0\}$.
Moreover, given the decomposition $Z^\p=N^\p D^\p$ we may replace $Z^\p$ by $N^\p$ in the enlargement formula (\ref{ef}) due to It\^{o}'s product formula: For any $M\in\mathcal{M}_{loc}(\p,\F_t)$, we have
\begin{equation}\label{ef2}
M_{t\wedge\sigma}-\int_0^{t\wedge\sigma}\frac{d\langle M,N^\p\rangle_s}{N^\p_s}\in\mathcal{M}_{loc}(\p,\G_t).
\end{equation}
\end{rem}

Note that if $Z^\p_t>0$ a.s.~for all $t\geq0$ one can write $Z^\p_t=N^\p_t\cdot \exp(-\Lambda^\p_t)$, where $\Lambda^\p_t:=-\ln(D^\p_t)$ is referred to as the intensity process in the credit risk literature. Therefore, the process $D^\p$ is of particular interest in credit risk modeling. However, the intensity as well as the process $D^\p$ depends on the underlying probability measure. Therefore, one may wonder whether there exist changes of probability measure under which the intensity process remains unchanged. The following theorem answers this question.

\begin{thm}\label{D}
Assume that $\rho>0$ ~$\p$-a.s.~and that $\left(\tilde{\rho}_t\right)$ is continuous. If $Z^\p=N^\p D^\p$ and $Z^\q=N^\q D^\q$ denote the It\^o-Watanabe decompositions of the Az\'ema supermartingales of $\sigma$ under $\p$ and $\q$, then $D_t^\p=D_t^\q$ a.s.~for all $t\geq0$.
\end{thm}

\begin{rem}
Intuitively, to affect the intensity of $\sigma$ via a change of probability measure the $(\G_t)$-Radon-Nikodym density $(\tilde{\rho}_t)$ should involve a stochastic integral with respect to the discontinuous martingale $\1_{\{\sigma\leq t\}}-\int_0^{t\wedge\sigma}\frac{dA^\p_s}{Z^\p_s}$. Indeed, the above theorem shows that a change of measure via a continuous $(\G_t)$-martingale will not change the intensity process. See also Theorem 6.3 in \cite{CJN}.
\end{rem}

In order to prove Theorem \ref{D}, we need to compute the process
\[h_t=\E^\p\left(\rho\1_{\{\sigma>t\}}|\F_t\right)=\E^\p\left(\tilde{\rho}_t\1_{\{\sigma>t\}}|\F_t\right).\]
This requires the knowledge about the behaviour of the process $(\tilde{\rho}_t)$ before time $\sigma$. Hence, the following representation result for bounded $(\G_t)$-martingales up to time $\sigma$ is very helpful. It is an immediate consequence of Th\'eor\`eme 5.12 and Lemme 5.15 of \cite{Jeulin}, cf.~also Theorem 3.1 in \cite{JSong}.

\begin{thm}\label{rep1}
For any bounded $\zeta\in\G_\sigma$ there exists a local $(\p,\F_t)$-martingale $M$ and a bounded $(\F_t)$-predictable process $K$ such that 
\[\E^\p(\zeta|\G_t)=M_t-\int_0^t\frac{d\langle M,Z^\p\rangle_s}{Z^\p_s}-\int_0^t\frac{K_s}{Z^\p_s}dA^\p_s\quad\text{on}\ \{\sigma>t\}.\]
Furthermore, if $t\mapsto\E^\p(\zeta|\G_t)$ is continuous almost surely (i.e.~it does not jump at $\sigma$), then $K\equiv0$.
\end{thm}

\begin{proof}
To prove the theorem one can do exactly the same computations as in the proof of Theorem 3.1 in \cite{JSong} without using any martingale representation property. Since we are only interested in the behaviour before time $\sigma$, we do not need the $(\mathcal{H}')$ hypothesis. 
\end{proof}

\begin{rem}
The assumption (AC) is not needed to obtain a characterization of any bounded $(\G_t)$-martingale before time $\sigma$, cf.~\cite{JSong}. The above formulation is however sufficient for our purposes.
\end{rem}

We are now ready to prove Theorem \ref{D}.

\begin{proof} {\it (of Theorem \ref{D}) }
For every $n\in\N$ set $\tau_n:=\inf\{t\geq0:\ \tilde{\rho}_t=n\}$. Then there exists for every $n\in\N$ an $(\F_t)$-stopping time $\nu_n$ such that $\tau_n\wedge\sigma=\nu_n\wedge\sigma$, cf.~Lemma \ref{pred}. Theorem \ref{rep1} applied to $\tilde{\rho}_{\tau_n\wedge \sigma}$ yields the existence of a local $(\p,\F_t)$-martingale $M^n$ such that for all $t\geq0$,
\begin{eqnarray*}
h_{t\wedge\nu_n}&=&
\E^\p\left(\left.\tilde{\rho}_{t\wedge\nu_n\wedge\sigma}\1_{\{\sigma>t\wedge\nu_n\}}\right|\F_{t\wedge\nu_n}\right)
=
\E^\p\left(\left.\tilde{\rho}_{t\wedge\tau_n\wedge\sigma}\1_{\{\sigma>t\wedge\nu_n\}}\right|\F_{t\wedge\nu_n}\right)\\
&=&\E^\p\left(\left.\left(M^n_{t\wedge\nu_n}-\int_0^{t\wedge\nu_n}\frac{d\langle Z^\p,M^n\rangle_s}{Z^\p_s}\right)\1_{\{\sigma>t\wedge\nu_n\}}\right|\F_{t\wedge\nu_n}\right)\\
&=&\E^\p\left(\left.\left(M^n_{t\wedge\nu_n}-\int_0^{t\wedge\nu_n}\frac{d\langle N^\p,M^n\rangle_s}{N^\p_s}\right)\1_{\{\sigma>t\wedge\nu_n\}}\right|\F_{t\wedge\nu_n}\right)\\
&=&\left(M^n_{t\wedge\nu_n}-\int_0^{t\wedge\nu_n}\frac{d\langle N^\p,M^n\rangle_s}{N^\p_s}\right)Z_{t\wedge\nu_n}^\p.
\end{eqnarray*}
Hence for all $t\geq0$,
\[\rho_{t\wedge\nu_n}N_{t\wedge\nu_n}^\q D_{t\wedge\nu_n}^\q=\rho_{t\wedge\nu_n}Z_{t\wedge\nu_n}^\q=h_{t\wedge\nu_n}=\left(M^n_{t\wedge\nu_n}-\int_0^{t\wedge\nu_n}\frac{d\langle N^\p,M^n\rangle_s}{N^\p_s}\right)N_{t\wedge\nu_n}^\p D_{t\wedge\nu_n}^\p.\]
Since $\rho>0$ almost surely, we have $\{Z^\p>0\}=\{h>0\}=\{Z^\q>0\}$. Moreover, the process
\begin{eqnarray*}
\left(M^n_{t\wedge\nu_n}-\int_0^{t\wedge\nu_n}\frac{d\langle N^\p,M^n\rangle_s}{N^\p_s}\right)N_{t\wedge\nu_n}^\p
=\qquad\qquad\qquad\qquad\qquad\qquad\qquad\qquad\qquad\qquad\\
\qquad\qquad M^n_0N^\p_0+\int_0^{t\wedge\nu_n}N_s^\p dM^n_s+\int_0^{t\wedge\nu_n}\left(M^n_s-\int_0^s\frac{d\langle N^\p,M^n\rangle_u}{N^\p_u}\right)dN^\p_s
\end{eqnarray*}
is a non-negative local $(\p,\F_t)$-martingale. Since $\left(\rho_{t\wedge\nu_n} N^\q_{t\wedge\nu_n}\right)$ is also a non-negative local $(\p,\F_t)$-martingale, the uniqueness of the It\^o-Watanabe decomposition yields that
\[\left(M^n_{t\wedge\nu_n}-\int_0^{t\wedge\nu_n}\frac{d\langle N^\p,M^n\rangle_s}{N^\p_s}\right)N_{t\wedge\nu_n}^\p=\rho_{t\wedge\nu_n}N_{t\wedge\nu_n}^\q\quad\text{and}\quad D_{t\wedge\nu_n}^\p=D_{t\wedge\nu_n}^\q\]
almost surely on $\{Z^\p>0\}=\{Z^\q>0\}$. Because $\tau_n\ra\infty$, we have
\[\sup_n(\nu_n\wedge\sigma)=\sup_n(\tau_n\wedge\sigma)=\sigma,\]
which implies that $\nu:=\sup_n\nu_n\geq\sigma$ almost surely. But then $Z^\p_\nu=0=Z^\q_\nu$ a.s.~and therefore $\nu\geq T_0^{Z^\p}= T_0^{Z^\q}$ $\p$-a.s. Because $D^\q$ and $D^\p$ are monotone increasing, the claim follows by sending $n\ra\infty$, noting that $D^\p$ and $D^\q$ are constant after time $T_0^{Z^\p}= T_0^{Z^\q}$.
\end{proof}

The following counterexample shows that the assumption that $(\tilde{\rho}_t)$ is continuous cannot be dropped in Theorem \ref{D}. 

\begin{ex}\label{nnn}
Let $N$ be a non-negative local martingale starting from $N_0=1$ and converging to zero almost surely and set $\sigma=\sup\{t>0:\ N_t=\ol{N}_t\}$. By Doob's maximal equality, cf.~Lemma 2.1 in \cite{DoobMax},
\[\p\left(\left.\sup_{s>t}N_s>a\right|\F_t\right)=\frac{N_t}{a}\quad\forall\ a>\ol{N}_t,\]
which implies that
\[Z_t^\p=\p\left(\left.\sup_{s>t}N_s>\ol{N}_t\right|\F_t\right)=\frac{N_t}{\ol{N}_t}=1+\int_0^t\frac{dN_s}{\ol{N}_s}-\log\left(\ol{N}_t\right).\]
Therefore, $supp(d\langle N\rangle)\subset\{Z^\p>0\}$ and $supp(d\ol{N})\subset\{Z^\p>0\}$. Hence, the uniqueness of the multiplicative decomposition defined in Theorem \ref{IW} implies that 
\[N^\p=N\quad\text{and}\quad D^\p=\frac{1}{\ol{N}}.\]
Now we may take $\rho=\log\left(\ol{N}_\infty\right)$ because $\E^\p\log\left(\ol{N}_\infty\right)=N_0\int_1^\infty\frac{da}{a}=1$. Then
\begin{eqnarray*}
h_t&=&\E^\p\left(\left.\log\left(\ol{N}_\infty\right)\1_{\left\{\sup_{s>t}N_s>\ol{N}_t\right\}}\right|\F_t\right)=
\E^\p\left(\left.\log\left(\sup_{s>t}N_s\right)\1_{\left\{\sup_{s>t}N_s>\ol{N}_t\right\}}\right|\F_t\right)\\
&=&N_t\int_{\ol{N}_t}^\infty\frac{\log(x)}{x^2}dx=\frac{N_t}{\ol{N}_t}(1+\log\left(\ol{N}_t)\right)
\end{eqnarray*}
and applying Lemma \ref{pred},
\begin{eqnarray*}
\tilde{\rho}_t=\E^\p(\rho|\G_t)&=&
\1_{\{\sigma\leq t\}}\log\left(\ol{N}_t\right)+\1_{\{\sigma>t\}}\frac{\E^\p\left(\left.\log\left(\ol{N}_\infty\right)\1_{\{\sigma>t\}}\right|\F_t\right)}{Z^\p_t}\\
&=&\1_{\{\sigma\leq t\}}\log\left(\ol{N}_t\right)+\1_{\{\sigma>t\}}\frac{\ol{N}_t}{N_t}\cdot h_t\\
&=&\1_{\{\sigma\leq t\}}\log\left(\ol{N}_t\right)+\1_{\{\sigma>t\}}\left(1+\log\left(\ol{N}_t\right)\right)= \log\left(\ol{N}_t\right)+\1_{\{\sigma>t\}}.
\end{eqnarray*}
Hence, $\tilde{\rho}$ is a purely discontinuous $(\p,\G_t)$-martingale and
\[\rho_t=\log\left(\ol{N}_t\right)+\frac{N_t}{\ol{N}_t}.\]
Therefore, 
\[Z_t^\q=\frac{h_t}{\rho_t}=\frac{N_t}{\ol{N}_t}\left(1+\log\left(\ol{N}_t\right)\right)\frac{1}{\rho_t}=\frac{N_t}{\ol{N}_t}\frac{\left(1+\log\left(\ol{N}_t\right)\right)\ol{N}_t}{N_t+\ol{N}_t\log\left(\ol{N}_t\right)}=\frac{N_t+N_t\log\left(\ol{N}_t\right)}{N_t+\ol{N}_t\log(\ol{N}_t)}.\]
And since $(N_t/\rho_t)\in\mathcal{M}_{loc}(\q,\F_t)$, the It\^o-Watanabe decomposition of $Z^\q$ takes the form
\[Z^\q_t=\frac{N_t}{\rho_t}\cdot\frac{1+\log(\ol{N}_t)}{\ol{N}_t}=N_t^\q D_t^\q\]
with
\[D_t^\q=\frac{1+\log(\ol{N}_t)}{\ol{N}_t}\neq\frac{1}{\ol{N}_t}=D_t^\p\quad\text{for $t$ large enough}.\]
\end{ex}

\s{Locally absolutely continuous change of measure}\label{abs}

In this section we slightly change the general setup introduced in section \ref{notations}. We will no longer rely on the existence of a random variable $\rho\geq0$ to define $\q$, but instead we will only assume the existence of some non-negative $(\p,\G_t)$-martingale $(\tilde{\rho}_t)$ with expectation one. As before $(\rho_t)$ is the $(\F_t)$-optional projection of $(\tilde{\rho}_t)$. Moreover, we will assume that $\F=\F_\infty:=\bigvee_{t\geq0}\F_t$ and that $(\Omega,\F,(\G_t)_{t\geq0},\p)$ is the natural augmentation of a probability space satisfying the Parthasarathy condition $(P)$, which can be found in the appendix.

For every $t\geq0$ we now define a probability measure $\q_t$ on $\G_t$ via $\q_t=\tilde{\rho}_t.\p|_{\G_t}$. This family of probability measures is consistent and since we assume our probability space to satisfy condition $(P)$ as well as the natural (but not the usual!) assumptions, Corollary 4.9 of \cite{newkind} yields the existence of a measure $\q$ on $\F=\G_\infty$ such that $\q|_{\G_t}=\q_t$ for all $t\geq0$. Note that $\q$ is only locally absolutely continuous with respect to $\p$, which we denote by $\q\triangleleft\p$. We define the process $h$ in this case by $h_t:=\E^\p(\tilde{\rho}_t\1_{\{\sigma>t\}}|\F_t)$. If $\q\ll\p$, this definition coincides with the one in section \ref{notations}. $\mu$ can now be defined as before.\vskip6pt

In this setting the following extended version of Theorem \ref{original} holds. 

\begin{thm}\label{lac}
Assume that $\q\triangleleft\p$. If $U=(U_t)_{t\geq0}$ is a local $(\p,\F_t)$-martingale, then the processes $X_t:=\1_{\{\sigma>t\}}V_t\exp(\mu^F_t)$ and $V_{t\wedge\sigma}$ are both local $(\q,\G_t)$-martingales, where $V_{t\wedge\sigma}:=U_{t\wedge\sigma}-\langle U,\mu\rangle_{t\wedge\sigma}$.
\end{thm}

\begin{proof}
Since $\q|_{\G_n}\ll\p|_{\G_n}$ the claim holds for every $U^n_t:=U_{t\wedge n}$ according to Theorem \ref{h0}. Especially, all processes are well-defined on $\bigcup_{n\in\N}[0,n]=\R_+$. But every process which is locally in $\mathcal{M}_{loc}(\q,\G_t)$ is actually a local martingale on the whole time interval. 
\end{proof}

The motivation to study locally absolutely continuous changes of measures comes from the fact that it may allow us to get rid off the random time $\sigma$ by pushing it to infinity as the following example demonstrates. 

\begin{ex}\label{push}
Consider
\[\tilde{\rho}_t=\frac{\1_{\{\sigma>t\}}}{Z_t^\p}.\]
This does indeed define a $(\G_t)$-martingale: for $s\leq t$,
\[\E^\p\left(\left.\frac{\1_{\{\sigma>t\}}}{Z_t^\p}\right|\G_s\right)=\frac{\1_{\{\sigma>s\}}}{Z_s^\p}\cdot\E^\p\left(\left.\frac{\1_{\{\sigma>t\}}}{Z_t^\p}\right|\F_s\right)=\frac{\1_{\{\sigma>s\}}}{Z_s^\p},\]
where we used Lemma \ref{pred} to compute the conditional expectation. Under the measure $\q$ defined as above $\sigma$ is pushed to infinity since 
\[\q(\sigma\leq t)=\E^\p\left(\tilde{\rho}_t\1_{\{\sigma\leq t\}}\right)=\E^\p\left(\frac{\1_{\{\sigma>t\}}}{Z_t^\p}\1_{\{\sigma\leq t\}}\right)=0\quad \forall\ t\geq0.\]
This is possible because $\tilde{\rho}_t\ra0$ $\p$-a.s.~and therefore $\q$ is not absolutely continuous to $\p$ on $\F=\G_\infty$. Thus, $\q$ puts only positive weight on those events taking place before $\sigma$. Moreover, if $Z$ is a bounded random variable which is $\F_t$-measurable for some $t\geq0$, then
\[\q(Z\leq x)=\p(Z\leq x)\quad\forall\ x\in\R,\] 
because $\rho_t\equiv 1$ for all $t\geq0$. Therefore,  $\F_t$-events do not "feel" the change of measure. Especially, 
any $(\p,\F_t)$-martingale is also a $(\q,\F_t)$-martingale and by Theorem \ref{lac} also a $(\q,\G_t)$-martingale since $h_t=\rho_t=1$ for all $t\geq0$ and $\sigma=\infty\ \q$-a.s.\\
Note that in the computation of pre-$\sigma$-events this measure change has the same impact as simply projecting down on $(\F_t)$. Indeed, every $\G_t$-measurable random variable is equal to an $\F_t$-measurable random variable before time $\sigma$, and for every $F_t\in\F_t$ one has
\[\E^\p\left(F_t\1_{\{\sigma>t\}}\right)=\E^\p\left(\frac{\1_{\{\sigma>t\}}}{Z_t^\p}\cdot F_tZ_t^\p\right)=\E^\q\left(F_tZ_t^\p\right)=\E^\p\left(F_tZ_t^\p\right).\]
\end{ex}

\subs{A change of measure which is equivalent to the enlargement formula}

As before we denote by $Z^\p=N^\p D^\p$ the It\^o-Watanabe decomposition of the Az\'ema supermartingale of $\sigma$. Under the assumption that $N^\p$ is a true martingale, we may set
\[\left.\frac{d\q}{d\p}\right|_{\G_t}=\tilde{\rho}_t=\frac{\1_{\{\sigma>t\}}}{D^\p_t}.\]
One easily checks that this defines a $(\G_t)$-martingale: for $s\leq t$,
\[\E^\p\left(\left.\frac{\1_{\{\sigma>t\}}}{D^\p_t}\right|\G_s\right)=
\frac{\1_{\{\sigma>s\}}}{Z_s^\p}\cdot\E^\p\left(\left.\frac{\1_{\{\sigma>t\}}}{D^\p_t}\right|\F_s\right)=
\frac{\1_{\{\sigma>s\}}}{Z_s^\p}\cdot\E^\p\left(\left.N^\p_t\right|\F_s\right)=\frac{\1_{\{\sigma>s\}}}{D^\p_s}.\]
As in Example \ref{push} we have $\q(\sigma<\infty)=0$ and hence any local $(\q,\F_t)$-martingale is also a local $(\q,\G_t)$-martingale. However, now
\[h_t=\rho_t=N^\p_t\]
is non-trivial and therefore the measure change will affect $(\p,\F_t)$-martingales according to the usual Girsanov theorem: given a local $(\p,\F_t)$-martingale $U$, the process
\[V_t:=U_t-\int_0^t\frac{d\langle N^\p,U\rangle_s}{N^\p_s}=U_t-\int_0^t\frac{d\langle Z^\p,U\rangle_s}{Z^\p_s},\quad t\geq0,\]
is a local $(\q,\G_t)$-martingale by Theorem \ref{lac}, since $V_t=V_{\sigma\wedge t}$ $\q$-almost surely for all $t\geq0$. Therefore, changing the measure in this way has the same effect as an application of the enlargement formula under $\p$ before time $\sigma$. This can be compared to \cite{Yoeurp}, where the enlargement formula was derived by passing to the so called F\"ollmer measure associated with $Z^\p$, and to the local solution method for enlargements of filtrations developed in \cite{Song}.\vskip2pt

Also note that in this setup we have for any $\F_t$-measurable random variable $F_t$,
\[\E^\p\left(F_t\1_{\{\sigma>t\}}\right)=\E^\q\left(F_tD^\p_t\right).\]
Since $D^\p$ is decreasing, one can interpret $D^\p_t$ as a discount factor in the above formula.

\begin{rem}
In \cite{Collin} the above measure change is applied to the valuation of defaultable securities via the reduced-form approach. However, in that paper the default time is directly modeled as a totally inaccessible stopping time without performing a progressive enlargement of filtration.
\end{rem}

The following example provides some intuition how the above measure change pushes $\sigma$ to infinity.

\begin{ex}
Consider the random time
\[\sigma=\sup\left\{t\geq0:\ N_t=\sup_{s\leq t}N_s\right\}=\sup\left\{t\geq0:\ \frac{1}{N_t}=\inf_{s\leq t}\frac{1}{N_s}\right\},\] 
where $N$ is supposed to be a non-negative $(\p,\F_t)$-martingale with $N_0=1$, converging towards zero almost surely. 
In this case $N^\p=N$, cf.~Example \ref{nnn}. If we take $(\tilde{\rho}_t)$ as above, the reciprocal of $N$ becomes a $\q$-martingale: for $s\leq t$,
\[\E^\q\left(\left.\frac{1}{N_t}\right|\F_s\right)=\frac{1}{\rho_s}\E^\p\left(\left.\frac{\rho_t}{N_t}\right|\F_s\right)=\frac{1}{N_s}.\]
However, $1/N$ does not converge to infinity but to zero under $\q$ because $\q$ is singular to $\p$ on $\F_\infty$. For all $\eps>0$ we have by dominated convergence as $t\ra\infty$,
\[\q\left(\frac{1}{N_t}>\eps\right)=\E^\p\left(N_t\1_{\{1/\eps>N_t\}}\right)\ra 0.\]
Therefore, $\sigma=\infty$ $\q$-a.s. 
\end{ex}

\begin{rem}
In the above computations we have assumed that $N^\p$ is a true martingale. If $N^\p$ is only a local $(\p,\F_t)$-martingale, analogous computations can be done if one defines $\q$ as the F\"ollmer measure associated with $(\tilde{\rho}_t)$. In this case the random time $\sigma$ is "replaced" under $\q$ by the explosion time of $(\tilde{\rho}_t)$, which is equal to the $(\F_t)$-stopping time $T_0^{D^\p}$ $\q$-almost surely.
\end{rem}

\s{Changes of measure for honest times}\label{hon}

In this section we focus on a special class of random times called honest times. The setup will be the same as described at the beginning of Section \ref{general}.

\begin{df}\label{defhon}
A random time $\sigma$ on $(\Omega,\F,(\F_t),\p)$ is called honest if for any $t>0$, $\sigma$ is equal to an $\F_t$-measurable random variable on $\{\sigma<t\}$.
\end{df}

\begin{rem}
Note that the definition of an honest time does not depend on the probability measure. It is shown in Proposition (5,1) of \cite{Jeulin} that if $\sigma$ is honest, then there exists an optional set $\Lambda$ such that $\sigma(\omega)=\sup\{t:\ (t,\omega)\in\Lambda\}$ on $\{\sigma<\infty\}$. Since under assumption (C) the optional and predictable $\sigma$-field are equal, we may assume w.l.o.g.~that the set $\Lambda$ is predictable. Moreover, $\p(\sigma=\infty)=0$ due to (A) and therefore $\sigma$ is the end of a predictable set in our setup.
\end{rem}

\subs{Change of measure after an honest time}

So far we were only concerned with changes of measure {\it up to} an arbitrary random time $\sigma$. Of course, we cannot expect an analogue of Corollary \ref{cor1} to hold {\it after} an arbitrary random time $\sigma$, because in general $(\F_t)$-semimartingales are not necessarily $(\G_t)$-semimartingales after time $\sigma$. However, if $\sigma$ is an honest time, then it is well-known that the semimartingale property is preserved when passing from $(\F_t)$ to $(\G_t)$, cf.~e.g.~Th\'eor\`eme (5,10) in \cite{Jeulin}. Hence, in this case one can expect to have an extension of Corollary \ref{cor1} to the whole time horizon. Our goal in this subsection is to prove this result by similar means as in \cite{Mortimer}, i.e.~{\it without} relying on any results from the theory of enlargements of filtrations.\\

{\bf For the rest of this section we suppose that $\sigma$ is an honest time.} As before we assume that there exists a non-negative random variable $\rho$ with expectation one and we set $\q=\rho.\p$.
We define the $(\p,\F_t)$-submartingale $k$ via 
\[k_t=\E^\p(\rho|\F_t)-h_t=\E^\p\left(\rho\1_{\{\sigma\leq t\}}|\F_t\right).\]
In the following we will use for fixed $u\geq0$ the notation 
\[\mathcal{M}^u(\p,\F_t)\]
to denote the class of processes which are $(\p,\F_t)$-martingales on the interval $[u,\infty)$. Moreover, for each $t\geq0$ we choose an $\F_t$-measurable random variable $\sigma_t$ which satisfies the requirement of Definition \ref{defhon}, 
i.e.~$\1_{\{\sigma<t\}}\sigma=\1_{\{\sigma<t\}}\sigma_t$.

\begin{lem}\label{l1}
Fix $u\geq0$ and let $Y$ be an $(\F_t)$-adapted process such that $(\1_{\{\sigma_t\leq u\}}k_tY_t)_{t\geq u}\in\mathcal{M}_{loc}^u(\p,\F_t)$. Then $Y_t\1_{\{\sigma\leq u\}}\in\mathcal{M}_{loc}^u(\q,\G_t)$. 
\end{lem}

\begin{proof} Because any $(\F_t)$-localizing sequence will also serve as a $(\G_t)$-localizing sequence, we only need to prove the martingale case. Recalling that $\sigma$ is an honest time which avoids stopping times, we have for any bounded test function $F_s\in\F_s,\ s\leq t,$ and $u\leq s\leq t$,
\begin{eqnarray*}
\E^\q(Y_t\1_{\{\sigma\leq u\}}F_s)&=&\E^\q(Y_t\1_{\{\sigma\leq t\}}\1_{\{\sigma_t\leq u\}}F_s)=
\E^\p(Y_t\rho\1_{\{\sigma\leq t\}}\1_{\{\sigma_t\leq u\}}F_s)\\&=&
\E^\p(Y_tk_t\1_{\{\sigma_t\leq u\}}F_s)=\E^\p(Y_sk_s\1_{\{\sigma_s\leq u\}}F_s)\\
&=&\E^\p(Y_s\rho\1_{\{\sigma\leq s\}}\1_{\{\sigma_s\leq u\}}F_s)=\E^\q(Y_s\1_{\{\sigma\leq u\}}F_s).
\end{eqnarray*}
Furthermore, if in addition $r\leq s$, then one gets
\begin{eqnarray*}
\E^\q(Y_t\1_{\{\sigma\leq u\}}\1_{\{\sigma\leq r\}}F_s)&=&\E^\q(Y_t\1_{\{\sigma\leq u\}}\1_{\{\sigma\leq s\}}\1_{\{\sigma_s\leq r\}}F_s)=\E^\q(Y_t\1_{\{\sigma\leq u\}}\1_{\{\sigma_s\leq r\}}F_s)\\
&=&\E^\q(Y_s\1_{\{\sigma\leq u\}}\1_{\{\sigma_s\leq r\}}F_s)=
\E^\q(Y_s\1_{\{\sigma\leq u\}}\1_{\{\sigma\leq s\}}\1_{\{\sigma_s\leq r\}}F_s)\\
&=&\E^\q(Y_s\1_{\{\sigma\leq u\}}\1_{\{\sigma\leq r\}}F_s).
\end{eqnarray*}
The monotone class theorem allows us to conclude that $Y_t\1_{\{\sigma\leq u\}}$ is a $\q$-martingale with respect to $\left(\F_t\vee\sigma(\1_{\{\sigma\leq r\}};r\leq t)\right)_{t\geq u}$. Because martingales with respect to some filtration remain martingales with respect to its right-continuous augmentation, we thus conclude that $Y_t\1_{\{\sigma\leq u\}}\in\mathcal{M}^u(\q,\G_t)$.
\end{proof}

\begin{rem}\label{rl1}
Note that if $(Y_t)_{t\geq u}$ is a martingale with respect to $\q$ and $(\G_t)_{t\geq u}$ on the set $\{\sigma\leq u\}$, then it is also a martingale on any $\G_u$-measurable subset of $\{\sigma\leq u\}$. Thus, for example 
\[\left(Y_t\1_{\{u_i\leq\sigma<u_j\}}\right)_{t\geq u}\in\mathcal{M}^u(\G_t,\q)\]
for every $0\leq u_i<u_j\leq u$.
\end{rem} 

\begin{lem}\label{l1b}
Let $(Y_t)$ be a real-valued continuous $(\G_t)$-adapted process such that the process $(\1_{\{\sigma\leq u\}}(Y_{t\vee u}-Y_u))_{t\geq 0}\in\mathcal{M}_{loc}(\q,\G_t)$ for all $u>0$. Then $\left(Y_{t\vee\sigma}-Y_\sigma\right)_{t\geq0}\in\mathcal{M}_{loc}(\q,\G_t)$. 
\end{lem}

\begin{proof}
Let us first assume that $Y$ is bounded. Then by Remark \ref{rl1} for all $u>v\geq0$,
\begin{eqnarray*}
\left(\1_{\{v\leq\sigma<u\}}(Y_{t\vee u}-Y_u)\right)_{t\geq 0}\in\mathcal{M}(\q,\G_t).
\end{eqnarray*}
We approximate $\sigma$ with the decreasing sequence of $(\G_t)$-stopping times 
\[s_n:=\sum_{k=1}^{n2^n}\frac{k}{2^n}\1_{\{(k-1)2^{-n}\leq\sigma<k2^{-n}\}}+\infty\1_{\{\sigma\geq n\}},\]
taking only finitely many values. Then for $s\leq t$ and $G_s\in\G_s$, because $Y$ is assumed to be bounded and continuous,
\begin{eqnarray*}
\E^\q\left((Y_{t\vee\sigma}-Y_\sigma)\1_{G_s}\right)&=&
\lim_{n\ra\infty}\E^\q\left((Y_{t\vee s_n}-Y_{s_n})\1_{G_s}\right)\\
&=&\lim_{n\ra\infty}\E^\q\left(\sum_{k=1}^{n2^n}\1_{\{s_n=k2^{-n}\}}(Y_{t\vee(k2^{-n})}-Y_{k2^{-n}})\1_{G_s}\right)\\
&=&\lim_{n\ra\infty}\sum_{k=1}^{n2^n}\E^\q\left(\underbrace{\1_{\{(k-1)2^{-n}\leq\sigma< k2^{-n}\}}(Y_{t\vee(k2^{-n})}-Y_{k2^{-n}})}_{\in\mathcal{M}(\q,\G_t)}\1_{G_s}\right)\\
&=&\lim_{n\ra\infty}\sum_{k=1}^{n2^n}\E^\q\left(\1_{\{(k-1)2^{-n}\leq\sigma<k2^{-n}\}}(Y_{s\vee(k2^{-n})}-Y_{k2^{-n}})\1_{G_s}\right)\\
&=&\lim_{n\ra\infty}\E^\q\left(\sum_{k=1}^{n2^n}\1_{\{s_n=k2^{-n}\}}(Y_{s\vee(k2^{-n})}-Y_{k2^{-n}})\1_{G_s}\right)\\
&=&\lim_{n\ra\infty}\E^\q\left((Y_{s\vee s_n}-Y_{s_n})\1_{G_s}\right)=\E^\q\left((Y_{s\vee\sigma}-Y_\sigma)\1_{G_s}\right),
\end{eqnarray*}
which proves that $Y_{t\vee\sigma}-Y_\sigma\in\mathcal{M}(\q,\G_t)$. Now the general case follows by localizing $Y$.
\end{proof}

\begin{thm}\label{exhon}
Let $\sigma$ be an honest time and suppose that $(U_t)_{t\geq0}$ is local $(\p,\F_t)$-martingale. Then the process 
\[V_t:=U_t-\int_0^{\sigma\wedge t}\frac{d\langle U,h\rangle_s}{h_s}-\int_{\sigma}^{\sigma\vee t}\frac{d\langle U,k\rangle_s}{k_s}\] 
is a local $(\q,\G_t)$-martingale.
\end{thm}

\begin{proof}
From Theorem \ref{h0} we already know that $(V_{t\wedge\sigma})_{t\geq0}$ is a local $(\q,\G_t)$-martingale. Therefore, it remains to show that $V_{t\vee\sigma}-V_{\sigma}\in\mathcal{M}_{loc}(\q,\G_t)$. According to Lemma \ref{l1b} this holds if for all $u>0$,
\[\1_{\{\sigma\leq u\}}(V_{t\vee u}-V_u)=\1_{\{\sigma\leq u\}}V_{t}^u\in\mathcal{M}_{loc}(\q,\G_t)\qquad\LRA\qquad \1_{\{\sigma\leq u\}}V_{t}^u\in\mathcal{M}^u_{loc}(\q,\G_t),\]
where we have defined for each $s\in\R_+$ the $(\F_t)$-adapted process
\[V_t^s:=U_{t\vee s}-U_s-\int_s^{t\vee s}\frac{d\langle k,U\rangle_u}{k_u}.\]
Therefore, an application of Lemma \ref{l1} will yield the result, if we can show that for all $u\geq0$, 
\[(\1_{\{\sigma_t\leq u\}}k_tV_t^u)_{t\geq u}\in\mathcal{M}^u_{loc}(\p,\F_t).\] 
First note that for all $t\geq u$, 
\[m_t^u:=\1_{\{\sigma_t\leq u\}}k_t=\E^\p(\rho\1_{\{\sigma\leq t,\sigma_t\leq u\}}|\F_t)=\E^\p(\rho\1_{\{\sigma\leq u\wedge t\}}|\F_t)=\E^\p(\rho\1_{\{\sigma\leq u\}}|\F_t)\]
and hence for every fixed $u>0$, $m^u\in\mathcal{M}^u(\p,\F_t)$. We apply integration by parts for $t\geq u$ to get
\begin{eqnarray*}
d(\1_{\{\sigma_t\leq u\}}k_tV_t^u)&=&d(m^u_tV_t^u)=V_t^udm^u_t+m^u_tdV_t^u+d\langle m^u,V^u\rangle_t\\
&=&V_t^udm^u_t+\1_{\{\sigma_t\leq u\}}\left[k_t\left(dU_t-\frac{d\langle k,U\rangle_t}{k_t}\right)+d\langle k,U\rangle_t\right]\\
&=&V_t^udm^u_t+m^u_tdU_t,
\end{eqnarray*}
which is an element of $\mathcal{M}^u_{loc}(\p,\F_t)$ for every $u>0$ as required. 
\end{proof}

\begin{rem}
In fact a more general version of Theorem \ref{exhon} is known to hold even without assuming (AC). This can be proven by applying first Girsanov's theorem and second the enlargement formula for honest times as it is done in paragraph 81 in \cite{DM5}. Note however, that our proof does \textit{not} make use of the enlargement formula. It only uses Definition \ref{defhon} of an honest time. Therefore as a byproduct by setting $\rho\equiv1$ we do actually recover the enlargement formula after $\sigma$ for honest times.
\end{rem}

\subs{Relative martingales}

Given an honest time $\sigma$ the process $k$ introduced in the previous subsection is actually a so called ``relative martingale''. Relative martingales were introduced in \cite{relmart} and are defined as follows.

\begin{df}
Let $\sigma$ be an honest time and $(Y_t)$ an $(\F_t)$-adapted right-continuous process such that $Y_\infty:=\lim_{t\ra\infty}Y_t$ exists $\p$-almost surely and in $L^1(\p)$. Then $(Y_t)$ is called a relative martingale associated with $\sigma$, if $Y_t=\E^\p(Y_\infty\1_{\{\sigma\leq t\}}|\F_t)$ for all $t\geq0$.
\end{df}

Hence, for an honest time $\sigma$ the process $k_t=\E^\p(\rho\1_{\{\sigma\leq t\}}|\F_t)$ is a relative martingale with final value $k_\infty=\E^\p(\rho|\F_\infty)$. Therefore, the class of relative martingales associated with $\sigma$ will provide us with nice non-trivial examples to illustrate Theorem \ref{exhon}. The following result from \cite{relmart} is very helpful in finding relative martingales. 

\begin{lem}\label{sigmad}
Let $(Y_t)$ be a continuous non-negative submartingale of class $(D)$ with Doob-Meyer decomposition $Y=M+F$, where $M\in\mathcal{M}_{loc}(\p,\F_t)$ and $F$ is an increasing $(\F_t)$-adapted process. Assume that $M_0=F_0=0$, $\p(Y_\infty=0)=0$ and that the measure $(dF_t)$ is carried by the set $\{t:\ Y_t=0\}$. Then $(Y_t)$ is a relative martingale associated with $\sigma=\sup\{t\geq0:\ Y_t=0\}$.
\end{lem}

Using Lemma \ref{sigmad} we can now give an example for an application of Theorem \ref{exhon}.

\begin{ex}
Let $B$ be a standard $(\p,\F_t)$-Brownian motion with $L$ denoting its local time at level zero. 
Set $\sigma=\sup\{\sigma\leq 1:\ B_t=0\}$. The submartingale
\[|B_{t\wedge 1}|=\int_0^{t\wedge1}\text{sgn}(B_u)dB_u+L_{t\wedge1}\]
fulfills the assumptions of Lemma \ref{sigmad} and is hence a relative martingale associated with $\sigma$. Setting $\rho=|B_1|$ we have for $t\leq 1$,
\begin{eqnarray*}
k_t&=&|B_{t}|=\int_0^{t}\text{sgn}(B_u)dB_u+L_t\\
\rho_t&=&\E^\p(\rho|\F_t)=\E^\p(|B_1||\F_t)=\int_{-\infty}^\infty\frac{|x+B_t|}{\sqrt{2\pi(1-t)}}\exp\left(-\frac{x^2}{2(1-t)}\right)dx\\
&=&|B_t|\cdot\left[2\Phi\left(\frac{|B_t|}{\sqrt{1-t}}\right)-1\right]+\sqrt{\frac{2(1-t)}{\pi}}\cdot\exp\left(-\frac{|B_t|^2}{2(1-t)}\right)\\
h_t&=&\rho_t-k_t=2|B_t|\cdot\left[\Phi\left(\frac{|B_t|}{\sqrt{1-t}}\right)-1\right]+\sqrt{\frac{2(1-t)}{\pi}}\cdot\exp\left(-\frac{|B_t|^2}{2(1-t)}\right)\\
dh_t&=&2\left[\Phi\left(\frac{|B_t|}{\sqrt{1-t}}\right)-1\right]\text{sgn}(B_t)dB_t+\text{finite variation part}.
\end{eqnarray*}
Thus according to Theorem \ref{exhon} the process
\[W_t:=B_t-\int_0^{t\wedge\sigma}\frac{\text{sgn}(B_s)\left[\Phi\left(\frac{|B_s|}{\sqrt{1-s}}\right)-1\right]ds}{|B_s|\cdot\left[\Phi\left(\frac{|B_s|}{\sqrt{1-s}}\right)-1\right]+\sqrt{\frac{1-s}{2\pi}}\cdot\exp\left(-\frac{|B_s|^2}{2(1-s)}\right)}
+\int_{t\wedge \sigma}^{t\wedge 1}\frac{ds}{B_s}\]
is a $(\q,\G_t)$-Brownian motion.
\end{ex}

\subs{An example related to the specific structure of the Az\'ema supermartingale of an honest time}

It follows from a result of Az\'ema (cf.~the Th\'eor\`eme on page 300 in \cite{Azema}) that for an honest time $\sigma$ the dual predictable projection $A^\p$ of the process $\left(\1_{\{\sigma\leq t\}}\right)_{t\geq0}$  satisfies under the assumptions (AC), 
\[supp\left(dA^\p\right)\subset\left\{Z^\p=1\right\}.\]

This property was used in \cite{DoobMax} to derive the general structure of the Az\'ema supermartingale associated with an honest time under (AC), cf.~Theorem 4.1 in \cite{DoobMax}:

\begin{lem}\label{Max}
For an honest time $\sigma$ there exists a non-negative local $(\p,\F_t)$-martingale $(N_t)_{t\geq0}$ with $N_0=1$ and $N_t\ra0$ $\p$-a.s.~such that 
\[Z_t^\p=\p(\sigma>t|\F_t)=\frac{N_t}{\ol{N}_t}.\]
\end{lem}

Therefore, by the same reasoning as in Example \ref{nnn}, the It\^o-Watanabe decomposition of $Z^\p$ is given by
\[Z_t^\p=N^\p_tD_t^\p\quad\text{with}\quad N^\p_t=N_t \quad\text{and}\quad D_t^\p=\frac{1}{\ol{N}_t}=\frac{1}{\ol{N}^\p_t},\quad t\geq0.\]

\begin{lem}\label{logN}
Let $\sigma$ be an honest time and denote by $Z_t^\p=N^\p_t/\ol{N}^\p_t$ the multiplicative decomposition of $Z_t^\p=\p(\sigma>t|\F_t)$ given in Lemma \ref{Max}. Then 
$A_\sigma^\p=A_\infty^\p$ and for all $t,x\geq0$,
\[\p\left(\left.A^\p_\sigma\in dx\right|\F_t\right)=N_t^\p e^{-x}dx\quad\text{on the set }\ \{x>A^\p_t\}.\]
\end{lem}

\begin{proof}
From Lemma 2.1 in \cite{DoobMax} we know that for $x>0$,
\[\p\left(\left.\sup_{s\geq t}N_s^\p>x\right|\F_t\right)=\left(\frac{N^\p_t}{x}\right)\wedge 1.\]
Moreover, it follows from Lemma \ref{Max} and It\^o's product formula that $A_t^\p=\log\left(\ol{N}^\p_t\right)$, which implies that 
\begin{eqnarray*}
\E^\p \left(A_\sigma^\p\right)&=&\E^\p\left(\int_0^\infty A_u^\p dA_u^\p\right)=\frac{1}{2}\cdot\E^\p \left(A_\infty^\p\right)^2
=\frac{1}{2}\cdot\E^\p \left(\log\left(\ol{N}^\p_\infty\right)\right)^2\\
&=&\frac{1}{2}\int_0^\infty \p \left(\log\left(\ol{N}^\p_\infty\right)^2>x\right)dx=\frac{1}{2}\int_0^\infty e^{-\sqrt{x}}dx=1=\E^\p\left(A^\p_\infty\right).
\end{eqnarray*}
Hence, since $A$ is non-decreasing, we must have $A_\sigma^\p=A_\infty^\p$ a.s. Therefore,
\[\p\left(\left.A^\p_\sigma>x\right|\F_t\right)=\p\left(\left.A^\p_\infty>x\right|\F_t\right)=\p\left(\left.\ol{N}^\p_\infty>e^x\right|\F_t\right)=\1_{\{\ol{N}^\p_t>e^x\}}+\1_{\{\ol{N}^\p_t\leq e^x\}}N_t^\p e^{-x}.\]
\end{proof}

Lemma \ref{logN} allows us to provide an example of an interesting class of measure changes, which - even though they may have different effects on a given $(\F_t)$-local martingale $U$ in the filtration $(\F_t)$ - yield the same $(\G_t)$-semimartingale decomposition of $U$ up to time $\sigma$:

\begin{ex}\label{hsup}
Suppose that $\sigma$ is an honest time and let $f:\R_+\ra\R_+$ be any measurable function such that $\int_0^\infty f(x)e^{-x}dx=1$. Then by Lemma \ref{logN},
\begin{eqnarray*}
h_t&=&\E^\p\left(\left.f\left(A^\p_\sigma\right)\1_{\{\sigma>t\}}\right|\F_t\right)=\E^\p\left(\left.\int_t^\infty f\left(A^\p_u\right)dA^\p_u\right|\F_t\right)= 
\E^\p\left(\left.\int_{A^\p_t}^{A^\p_\infty} f(x)dx\right|\F_t\right)\\
&=&\E^\p\left(\left.\int_{A^\p_t}^{A^\p_\sigma} f(x)dx\right|\F_t\right)=N_t^\p\int_{A^\p_t}^\infty\int_{A^\p_t}^yf(x)dx\ e^{-y}dy=
N_t^\p\int_{A^\p_t}^\infty\int_{x}^\infty e^{-y}dyf(x)dx\\ 
&=&N_t^\p\int_{A^\p_t}^\infty f(x)e^{-x}dx,
\end{eqnarray*}
\begin{eqnarray*}
dh_t&=&\int_{A^\p_t}^\infty f(x)e^{-x}dx\ dN^\p_t-N^\p_tf\left(A^\p_t\right)e^{-A^\p_t}dA^\p_t,\\
d\mu_t&=&\frac{dh_t}{h_t}=\frac{dN^\p_t}{N^\p_t}-\frac{f\left(A^\p_t\right)e^{-A^\p_t}dA^\p_t}{\int_{A^\p_t}^\infty f(x)e^{-x}dx}=\frac{dZ^\p_t}{Z^\p_t}+d\log\left(\int_{A^\p_t}^\infty f(x)e^{-x}dx\right),\\
k_t&=&\E^\p\left(\left.f\left(A^\p_\sigma\right)\1_{\{\sigma\leq t\}}\right|\F_t\right)=\E^\p\left(\left.f\left(A^\p_t\right)\1_{\{\sigma\leq t\}}\right|\F_t\right)=f\left(A^\p_t\right)\left(1-Z_t^\p\right),\\
dk_t&=&-f\left(A^\p_t\right)dZ_t^\p+\left(1-Z_t^\p\right)d\left[f\left(A^\p_t\right)\right]=-f\left(A^\p_t\right)dZ_t^\p.
\end{eqnarray*}
Applying Theorem \ref{exhon} we see that given a continuous local $(\p,\F_t)$-martingale $U$ the process 
\[V_{t\wedge\sigma}:=U_{t\wedge\sigma}-\int_0^{t\wedge\sigma}\frac{d\langle Z^\p,U\rangle_s}{Z^\p_s}+\int_\sigma^{\sigma\vee t}\frac{d\langle Z^\p,U\rangle_s}{1-Z_s^\p},\quad t\geq0,\] 
is a local $(\q,\G_t)$-martingale for all measures $\q$ defined as above and associated with a function $f$ satisfying $\int_0^\infty f(x)e^{-x}dx=1$. We thus note that the semimartingale decomposition of $U$ in the filtration $(\G_t)$ does {\it not} depend on $f$. 
However, note that the $(\F_t)$-Radon-Nykodim density of $\q$ with respect to $\p$ does depend on $f$:
\[\rho_t=h_t+\E^\p\left(\left.f\left(A_\sigma^\p\right)\1_{\{\sigma\leq t\}}\right|\F_t\right)=N_t^\p\int_{A^\p_t}^\infty f(x)e^{-x}dx+f\left(A_t^\p\right)\left(1-Z_t^\p\right).\]
\end{ex}

\s{Changes of measure up to pseudo-stopping times}\label{pst}

In this section we focus on a special class of random times called pseudo-stopping times, which were introduced in \cite{pseudo}.

\begin{df}
A positive random variable $\sigma:(\Omega,\F)\ra(\R_+,\mathcal{B}(\R_+))$ is called a $(\p,\F_t)$-pseudo-stopping time if $\E^\p M_\sigma=\E^\p M_0$ for every uniformly integrable \mbox{$(\p,\F_t)$-martingale $M$.}
\end{df}

In \cite{pseudo} it is shown that pseudo-stopping times can be characterized in many different ways:

\begin{thm}\label{pseudo}
The following are equivalent:
\be
\item [(1)] $\sigma$ is a $(\p,\F_t)$-pseudo stopping time.
\item [(2)] $A^\p_\infty\equiv1$ almost surely.
\item [(3)] $A^\p_\sigma\sim\mathcal{U}[0,1]$.
\item [(4)] For any local $(\p,\F_t)$-martingale $M=(M_t)_{t\geq0}$, the process $(M_{t\wedge\sigma})_{t\geq0}$ is a local $(\p,\G_t)$-martingale.
\item [(5)] $Z^\p=1-A^\p$ is a decreasing $(\F_t)$-predictable process.
\ee
\end{thm}

\begin{proof}
The equivalence between (1), (2), (4) and (5) is shown in Theorem 1 of \cite{pseudo}, while the implication (1)$\RA$(3) is a direct consequence of Proposition 2 of \cite{pseudo}. Finally, the relation $(2)\LRA(3)$ follows from the general relation between the Laplace transforms of $A^\p_\sigma$ and $A^\p_\infty$: Indeed, since $(A^\p_t)$ is the dual predictable projection of $(\1_{\{\sigma\leq t\}})$, we have
\[\lambda\cdot \E^\p\left( e^{-\lambda A^\p_\sigma}\right)=\lambda\cdot\E^\p\left(\int_0^\infty e^{-\lambda A^\p_u}dA^\p_u\right)
=1-\E^\p \left(e^{-\lambda A^\p_\infty}\right),\quad\lambda>0.\]
\end{proof}

\subs{First results}

We immediately derive the following lemma.

\begin{lem}\label{M}
Let $\sigma$ be a $(\p,\F_t)$-pseudo-stopping time and suppose that $\rho=M_\sigma$, where $M$ is a strictly positive uniformly integrable $(\p,\F_t)$-martingale starting from $M_0=1$. Then
\[V_{t\wedge\sigma}=U_{t\wedge\sigma}-\int_0^{t\wedge\sigma}\frac{d\langle M,U\rangle_s}{M_s}\] 
is a local $(\q,\G_t)$-martingale.
\end{lem}

\begin{proof}
By the pseudo-stopping time property, $\E^\p M_\sigma=\E^\p M_0=1$. Thus, $\rho$ is well-defined. Moreover due to part (4) of Theorem \ref{pseudo}, given any local $(\F_t)$-martingale $U$ the process $(U_{t\wedge\sigma})$ is a local $(\G_t)$-martingale. Especially, $(M_{t\wedge\sigma})$ is a local $(\G_t)$-martingale closed by $M_\sigma$ and hence a uniformly integrable $(\G_t)$-martingale. Therefore the usual Girsanov applied in the enlarged filtration implies that
\[V_{t\wedge\sigma}=U_{t\wedge\sigma}-\int_0^{t\wedge\sigma}\frac{d\langle M,U\rangle_s}{M_s}\] 
is a local $(\q,\G_t)$-martingale.
\end{proof}

\begin{rem}
Alternatively, Lemma \ref{M} can also be proven by applying Corollary \ref{cor1} with
\[h_t=\E^\p\left(M_\sigma\1_{\{\sigma>t\}}|\F_t\right)=\E^\p\left(M_{\sigma\wedge t}\1_{\{\sigma>t\}}|\F_t\right)\\
=M_t\cdot\p(\sigma>t|\F_t)=M_t(1-A_t^\p).\]
Note that we cannot choose $\rho=M_\infty$ instead of $\rho=M_\sigma$ in Lemma \ref{M} because in general  $\E^\p(M_\infty|\G_\sigma)\neq M_\sigma$ unless $\sigma$ is a stopping time. Also, generally $\rho_t\neq M_t$, i.e.~$M$ is {\it not} the Radon-Nikodym density of $\q$ with respect to $\p$ in the filtration $(\F_t)$. 
\end{rem}

The next example generalizes Example 2 in \cite{Mortimer} and provides us with a class of non-trivial measure changes up to a pseudo-stopping time $\sigma$ which do not affect $(\F_t)$-martingales. 

\begin{ex}\label{pst1}
Let $\sigma$ be a $(\p,\F_t)$-pseudo-stopping time and let $f:\R_+\ra\R_+$ be any measurable function satisfying $\int_0^1f(x)dx=1$. We choose $\rho=f\left(A_\sigma^\p\right)$. Then 
\begin{eqnarray*}
h_t&=&\E^\p\left(f\left(A_\sigma^\p\right)\1_{\{\sigma>t\}}|\F_t\right)=\E^\p\left(\left.\int_t^\infty f\left(A^\p_u\right)dA_u^\p\right|\F_t\right)=\int^1_{A_t^\p} f(x)dx,\\
dh_t&=&-f\left(A_t^\p\right)dA_t^\p,\\
d\mu_t&=&\frac{dh_t}{h_t}=\frac{-f\left(A_t^\p\right)dA_t^\p}{\int_{A_t^\p}^1 f(y)dy}=-d\mu_t^F.
\end{eqnarray*}
By Corollary \ref{cor1}, for every continuous local $(\p,\F_t)$-martingale $U$ the process $(U_{t\wedge\sigma})_{t\geq0}$ is a local $(\q,\G_t)$-martingale because $\langle U,\mu\rangle_{t\wedge\sigma}=-\langle U,\mu^F\rangle_{t\wedge\sigma}=0$ for all \mbox{$t\geq0$.} 
Note that this particular choice of $\rho$ does not have any effect on continuous $(\F_t)$-martingales until time $\sigma$: $(U_{t\wedge\sigma})$ is a local $(\q,\G_t)$-martingale for {\it any} choice of $f$ satisfying $\int_0^1f(x)dx=1$.
\end{ex}

\vskip6pt

As opposed to honest times, the definition of a $\p$-pseudo-stopping time depends on the underlying probability measure. The following example shows that the pseudo-stopping time property may indeed get lost under an equivalent change of probability measure.

\begin{ex}
Let $\sigma$ be an $\F_\infty$-measurable $(\p,\F_t)$-pseudo-stopping time and define the random variable $\rho=2Z_\sigma^\p\in\F_\infty$.
Since $Z_\sigma^\p\sim\mathcal{U}[0,1]$ by Theorem \ref{pseudo}, the measure $\q=\rho.\p$ is well-defined and equivalent to $\p$. We have
\begin{eqnarray*}
\tilde{\rho}_t&=&2\E^\p\left(\left.Z_\sigma^\p\right|\G_t\right)=2\1_{\{\sigma\leq t\}}Z_\sigma^\p+2\1_{\{\sigma>t\}}\frac{\E^\p\left(Z_\sigma^\p\1_{\{\sigma>t\}}|\F_t\right)}{Z_t^\p}\\
&=&2\1_{\{\sigma\leq t\}}Z^\p_\sigma+2\1_{\{\sigma>t\}}\frac{\E^\p\left(\left.\int_t^\infty \left(1-A_u^\p\right)dA_u^\p\right|\F_t\right)}{Z_t^\p}\\
&=&2\1_{\{\sigma\leq t\}}Z_\sigma^\p+\1_{\{\sigma>t\}}\frac{\left(1-A_t^\p\right)^2}{Z_t^\p}=
Z^\p_{t\wedge\sigma}+\1_{\{\sigma\leq t\}}Z^\p_\sigma,
\end{eqnarray*}
which jumps at time $\sigma$. Moreover,
\begin{eqnarray*}
h_t=\E^\p\left(\tilde{\rho}_t\1_{\{\sigma>t\}}|\F_t\right)=Z_t^\p\cdot\E^\p\left(\1_{\{\sigma>t\}}|\F_t\right)=\left(Z_t^\p\right)^2.
\end{eqnarray*}
Since $\rho=\E^\p(\rho|\F_\infty)=\rho_\infty\neq1$ almost surely, the continuous uniformly integrable martingale $(\rho_t)$ is not identical to one. Therefore, having in mind that $Z^\p$ is of finite variation,
\[Z_t^\q=\frac{h_t}{\rho_t}=\frac{\left(Z_t^\p\right)^2}{\rho_t}\]
cannot be of finite variation, which implies that $\sigma$ is \textit{not} a $\q$-pseudo-stopping time.
\end{ex}

However, the pseudo-stopping time property of $\sigma$ is in many cases very desirable, since local martingales remain local martingales in the progressively enlarged filtration until time $\sigma$ and therefore semimartingale decompositions do not change. Hence, it is an interesting mathematical question whether there exist equivalent changes of probability measures which preserve the pseudo-stopping time property. \vskip6pt

The following example of such a measure change is taken from the credit risk literature and is known as the so called Cox construction.

\begin{ex}\label{Cox}
Assume that there exists a random variable $U$ which is independent of $\F_\infty$ such that $\p(U>t)=\exp(-t)$ for all $t\geq0$. Let $(\Lambda_t)$ be an $(\F_t)$-adapted continuous increasing process with $\Lambda_\infty=\infty$ a.s.~and define
\[\sigma:=\inf\{t\geq0:\ \Lambda_t\geq U\}.\]
Then $\sigma$ is a $\p$-pseudo-stopping time because
\[Z_t^\p=\p(\sigma> t|\F_t)=\p(\Lambda_t< U|\F_t)=\exp(-\Lambda_t).\]
Let $\rho\in\F_\infty$ be a strictly positive random variable with $\E^\p\rho=1$, defining the equivalent measure $\q:=\rho.\p$. Then 
\begin{eqnarray*}
Z_t^\q&=&\frac{h_t}{\rho_t}=\frac{\E^\p(\E^\p(\rho\1_{\{\sigma>t\}}|\F_\infty)|\F_t)}{\rho_t}=
\frac{\E^\p(\rho\cdot \p(\Lambda_t<U|\F_\infty)|\F_t)}{\rho_t}\\
&=&\frac{\E^\p(\rho\exp(-\Lambda_t)|\F_t)}{\rho_t}=\exp(-\Lambda_t).
\end{eqnarray*}
Hence, $\sigma$ is a $\p$- and $\q$-pseudo-stopping time. Moreover, $Z^\q=Z^\p$ almost surely. In fact, one can show that in this case the filtration $(\F_t)$ is immersed in the larger filtration $(\G_t)$. 
In a financial setting, if we assume that the stock price is an $(\F_t)$-adapted process, then any equivalent local martingale measure (cf.~Def.~\ref{dual} below) can indeed be identified with some $\rho\in\F_\infty$ and will hence remain an equivalent martingale measure in the enlarged filtration if $\sigma$ is constructed as above. 
\end{ex}

Characterizing for an arbitrary $\p$-pseudo stopping time the class of measures changes that preserve the pseudo-stopping time is a quite challenging and unsolved problem, which is beyond the scope of this paper. However, in the next subsection we will approach this problem for a certain subclass of pseudo-stopping times and construct a non-trivial class of measure changes that have the desired property.

\subs{A special class of pseudo-stopping times}

In section 3.2 of \cite{pseudo} the authors provide a systematic construction of pseudo-stopping times under the assumptions (AC):

\begin{lem}\label{inf}
Let $L$ be an honest time with associated Az\'ema supermartingale $Z^L$ under $\p$. Then
\[\sigma:=\sup\left\{t<L:\ Z_t^L=\ul{Z}_L^L\right\}=\sup\left\{t<L:\ Z_t^L=\ul{Z}_t^L\right\} \]
is a $\p$-pseudo-stopping time and its Az\'ema supermartingale is given by
\[Z^\sigma_t:=\p(\sigma>t|\F_t)=\ul{Z}_t^L,\quad t\geq0.\]
\end{lem}

In this chapter we will only consider pseudo-stopping times which are constructed via Lemma \ref{inf}. The advantage of this is that we have more specific knowledge about the structure of the associated Az\'ema supermartingale. Let us look at an example.

\begin{ex}\label{b}
We suppose that the filtered probability space $(\Omega,\F,(\F_t),\p)$ satisfies the Parthasarathy condition $(P)$ and that there exists a $(\p,\F_t)$-Brownian motion $B$ on it. Let us define for all $a\in\R$ and $s\geq0$ the stopping time
\[\tau_s^a:=\inf\{t>s:\ B_t=a\}\]
as well as the random times
\[L:=\sup\{t<\tau_0^1:\ B_t=0\},\quad \sigma:=\sup\{t<L:\ \ol{B}_t=B_t\}=\sup\{t<L:\ \ol{B}_L=B_t\}.\]
Then 
\[Z_t^{\p,L}:=\p(L>t|\F_t)=1-B^+_{t\wedge\tau_0^1},\] 
cf.~VI.7.12 in \cite{RW}, which implies that $\sigma$ is a $(\p,\F_t)$-pseudo-stopping time, cf.~Lemma \ref{inf}. Indeed, this is the original example of a pseudo-stopping time provided by D.~Williams, cf.~\cite{Williams}. Let $b:\R\ra\R$ be a bounded function and set
\[\rho_t=\mathcal{E}\left(\int_0^tb(B_s)dB_s\right).\]
By Novikov's criterion $(\rho_t)_{t\geq0}$ is a positive $(\p,\F_t)$-martingale which defines - due to condition $(P)$  - a measure $\q$ on $\F_\infty$ such that
\[\left.\frac{d\q}{d\p}\right|_{\F_t}=\rho_t,\quad\forall\ t\geq0.\]
Note that in general $\q$ is only locally equivalent to $\p$, i.e.~it may be singular to $\p$ on $\F_\infty$. 
By Girsanov's theorem the process
\[W_t:=B_t-\int_0^tb(B_s)ds\]
is a $\q$-Brownian motion and $B$ is an It\^o-diffusion. We denote its $\q$-scale function by $s(\cdot)$. 
Using the Markov property of $B$ we compute the $\q$-Az\'ema supermartingale of $L$ as
\begin{equation}\label{2}
Z^{L,\q}_t:=\q(L>t|\F_t)=\1_{\{\tau_0^1>t\}}\q\left(\tau_t^1>\tau_t^0|\F_t\right)=\frac{s(1)-s\left(B_{t\wedge\tau_0^1}\right)}{s(1)-s(0)}\wedge 1.
\end{equation}
Since $s$ is a strictly increasing function, 
\begin{eqnarray*}
\sigma=\sup\{t<L:\ \ol{B}_L=B_t\}=\sup\left\{t<L:\ s\left(\ol{B}_L\right)=s(B_t)\right\}=\sup\left\{t<L:\ \ul{Z}_L^{L,\q}=Z_t^{L,\q}\right\}.
\end{eqnarray*}
According to Lemma \ref{inf}, $\sigma$ is thus also a $\q$-pseudo-stopping time.
\end{ex}

In the above example the measure change is chosen such that the Az\'ema supermartingale of the honest time $L$ under $\q$ is a monotone transformation of the Az\'ema supermartingale of $L$ under $\p$, cf.~equation (\ref{2}). This ensures that they both attain their infimum at the same time. Hence, the $\p$-pseudo stopping time associated with $L$ via Lemma \ref{inf} is identical to the $\q$-pseudo stopping time associated with $L$ via Lemma \ref{inf}. However, the construction is tailor-made for this specific example dealing with homogeneous diffusions and cannot easily be generalized. \\

The following theorem provides a class of measure changes which preserve the pseudo-stopping time property of a $\p$-pseudo-stopping time constructed via Lemma \ref{inf}. The idea is again that under the new measure $\q$ the Az\'ema supermartingale of the underlying honest time $L$ should be a monotone transformation of the Az\'ema supermartingale of $L$ under the measure $\p$.

\begin{thm}\label{gh}
Let $L$ be an honest time with Az\'ema supermartingale \mbox{$Z_t^{\p,L}:=\p(L>t|\F_t)$} and define the $\p$-pseudo-stopping time
\[\sigma:=\sup\left\{t<L:\ Z_t^{\p,L}=\ul{Z}^{\p,L}_L\right\}.\]
Moreover, let $g:[0,1]\ra\R$ be a Lebesgue-integrable function which satisfies 
\[\int_0^1\exp\left(\int_z^1g(y)dy\right)dz=1.\] 
We may write $Z_t^{\p,L}=N^\p_t/\ol{N}^\p_t$ for some non-negative local $\p$-martingale $N^\p$ with $N^\p_0=1$, converging to zero almost surely. If the process
\[\rho_t:=\mathcal{E}\left(\int_0^tg\left(\frac{N^\p_s}{\ol{N}^\p_s}\right)\frac{dN^\p_s}{2\ol{N}^\p_s}\right)\]
is a uniformly integrable $(\p,\F_t)$-martingale, then  $\sigma$ is also a $\q$-pseudo-stopping time with respect to the measure $\q:=\rho_\infty.\p$. 
\end{thm}

\begin{proof}
We set $\q=\rho_\infty.\p$ and define
\[M_t:=h\left(\frac{N^\p_t}{\ol{N}^\p_t}\right)\cdot\ol{N}^\p_t,\quad t\geq0,\]
where $h:[0,1]\ra\R_+$ is the function
\[h(x)=\int_0^x\exp\left(\int^1_zg(y)dy\right)dz.\]
Note that $h$ satisfies
\[g(x)h'(x)+h''(x)=0, \quad h(1)=h'(1)=1, \quad h(0)=0.\]
This implies that $M$ is a local $(\q,\F_t)$-martingale. Indeed, by Girsanov's theorem
\[\widetilde{N}_t:=N^\p_t-\int_0^tg\left(\frac{N^\p_s}{\ol{N}^\p_s}\right)\frac{d\left\langle N^\p\right\rangle_s}{2\ol{N}^\p_s}\]
is a local $(\q,\F_t)$-martingale and 
\begin{eqnarray*}
dM_t&=&h\left(\frac{N^\p_t}{\ol{N}^\p_t}\right)d\ol{N}^\p_t+\ol{N}^\p_th'\left(\frac{N_t^\p}{\ol{N}^\p_t}\right)\left[\frac{dN^\p_t}{\ol{N}^\p_t}-\frac{d\ol{N}^\p_t}{\ol{N}^\p_t}\right]+\frac{1}{2}h''\left(\frac{N^\p_t}{\ol{N}^\p_t}\right)\frac{d\left\langle N^\p\right\rangle_t}{\ol{N}^\p_t}\\
&=&[h(1)-h'(1)]d\ol{N}^\p_t+h'\left(\frac{N^\p_t}{\ol{N}^\p_t}\right)d\widetilde{N}_t=h'\left(\frac{N^\p_t}{\ol{N}^\p_t}\right)d\widetilde{N}_t,
\end{eqnarray*}
because $supp\left(d\ol{N}^\p\right)\subset\left\{N^\p=\ol{N}^\p\right\}$. Furthermore, $h$ is strictly increasing with $h(1)=1$. Therefore, $\ol{N}^\p=\ol{M}$ and
\[L=\sup\left\{t>0:\ N^\p_t=\ol{N}^\p_t\right\}=\sup\left\{t>0:\ M_t=\ol{M}_t\right\}.\]
Since $M_t\ra0$ almost surely, 
\[Z^{\q,L}_t:=\q(L>t|\F_t)=\frac{M_t}{\ol{M}_t}=h\left(\frac{N^\p_t}{\ol{N}^\p_t}\right)\frac{\ol{N}^\p_t}{\ol{M}_t}=h\left(\frac{N^\p_t}{\ol{N}^\p_t}\right)=h\left(Z_t^{\p,L}\right).\]
But then
\[\sigma=\sup\left\{t<L:\ Z^{\p,L}_t=\ul{Z}^{\p,L}_L\right\}=\sup\left\{t<L:\ Z^{\q,L}_t=\ul{Z}^{\q,L}_L\right\}\]
and $\sigma$ is a $\q$-pseudo stopping time by Lemma \ref{inf}.
\end{proof}

We now give an example of a function $g$ which fulfills the integrability condition required in Theorem \ref{gh}. 

\begin{ex}
Consider $g(x)=x-c$, where $c>0$ is chosen such that 
\[\int_0^1\exp\left(\frac{1-z^2}{2}-c(1-z)\right)dz=1.\] 
Using product integration,
\begin{eqnarray*}
\int_0^t\frac{N^\p_s}{\left(\ol{N}^\p_s\right)^2}dN^\p_s&=&
\left(\frac{N^\p_t}{\ol{N}^\p_t}\right)^2-1-\int_0^tN^\p_s\left(\frac{dN^\p_s}{\left(\ol{N}^\p_s\right)^2}-\frac{2N^\p_s}{\left(\ol{N}^\p_s\right)^3}d\ol{N}^\p_s\right)-\int_0^t\frac{d\left\langle N^\p\right\rangle_s}{\left(\ol{N}^\p_s\right)^2}\\
&\leq& -\int_0^t\frac{N^\p_s}{\left(\ol{N}^\p_s\right)^2}dN^\p_s+2\int_0^t\frac{d\ol{N}^\p_s}{\ol{N}^\p_s}\\
\LRA\quad\int_0^t\frac{N^\p_s}{\left(\ol{N}^\p_s\right)^2}dN^\p_s&\leq&\log\left(\ol{N}^\p_t\right),\\
X_t&:=&\int_0^t\frac{dN^\p_s}{\ol{N}^\p_s}=\frac{N^\p_t}{\ol{N}^\p_t}-1+\int_0^t\frac{N^\p_sd\ol{N}^\p_s}{\left(\ol{N}^\p_s\right)^2}=\frac{N^\p_t}{\ol{N}^\p_t}-1+\log\left(\ol{N}^\p_t\right)\geq -1,\\
Y_t&:=&\int_0^tg\left(\frac{N^\p_s}{\ol{N}^\p_s}\right)\frac{dN^\p_s}{\ol{N}^\p_s}=\int_0^t\frac{N^\p_sdN^\p_s}{\left(\ol{N}^\p_s\right)^2}-c\int_0^t\frac{dN^\p_s}{\ol{N}^\p_s}\leq c+\log\left(\ol{N}^\p_t\right)\\
&\leq& c+\log\left(\ol{N}^\p_\infty\right).
\end{eqnarray*}
First note that $X=(X_t)$ is a uniformly integrable martingale bounded from below, since
\[\E^\p X_\infty=0-1+\E^\p\log\left(\ol{N}^\p_\infty\right)=0-1+1=0,\]
where we have used the fact that $\log\left(\ol{N}^\p_\infty\right)\sim\text{Exp}(1)$, cf.~Lemma \ref{logN}. Moreover,
\[\sup_{t\geq0}\ \E^\p X_t^2\leq \E^\p\left(1+\log\left(\ol{N}^\p_\infty\right)\right)^2=\int_0^\infty(1+x)^2e^{-x}dx=5.\]
Therefore $X$ is square-integrable and
\[\E^\p\left\langle Y\right\rangle_\infty=
\E^\p\int_0^\infty g^2\left(\frac{N^\p_t}{\ol{N}^\p_t}\right)d\langle X\rangle_t\leq (1+c)^2\cdot\E^\p\langle X\rangle_\infty<\infty.\]
By the Burkholder-Davis-Gundy inequality thus $\E^\p \sup_{t\geq0}|Y_t|<\infty$ and the dominated convergence theorem yields the martingality of $Y=(Y_t)$. Moreover for all $t\geq0$,
\[\E^\p\exp\left(\frac{Y_t}{2}\right)\leq e^{c/2}\cdot\E^\p\exp\left(\frac{\log\left(\ol{N}^\p_\infty\right)}{2}\right)
=e^{c/2}\cdot\E^\p\sqrt{\ol{N}^\p_\infty}=e^{c/2}\cdot\int_0^1\frac{dx}{\sqrt{x}}=2e^{c/2}.\]
Hence, by Jensen's inequality $(\exp(Y_t/2))_{t\geq0}$ is a uniformly integrable submartingale and Kazamaki's criterion implies the uniform integrability of $(\rho_t)$.
\end{ex}

\s{Generalization of Example 1 from \cite{Mortimer}}\label{generalize}

In this section we come back to \cite{Mortimer}, which was the starting point of this paper. The goal is to
generalize Example 1 of \cite{Mortimer}, which is related to the path decomposition of the Brownian motion. Since in that example $\sigma$ is an honest time, we are able to extend the measure change beyond time $\sigma$ using Theorem \ref{exhon}. Moreover, in \cite{Mortimer} the authors do a Markovian study of their example. However, as it turns out their example is related to the construction in Lemma \ref{inf}, which allows us to look at it from a different angle and to extend it to general honest times. The construction is as follows:\\

Let $\sigma$ be an honest time. In this subsection its Az\'ema supermartingale with respect to $\p$ will be denoted by $Z^\sigma_t=\p(\sigma>t|\F_t)$ with Doob-Meyer decomposition $Z^\sigma_t=m^\sigma_t-A^\sigma_t$. 
Then by Lemma \ref{inf},
	\[\pi=\sup\left\{t<\sigma:\ Z_t^\sigma=\ul{Z}_\sigma^\sigma\right\} \]
is a $\p$-pseudo-stopping time and $Z_t^\pi:=\p(\pi>t|\F_t)=\inf_{u\leq t}Z^\sigma_u=\ul{Z}^\sigma_t=:1-A^\pi_t$ for all $t\geq0$.
From Theorem \ref{pseudo} we know that $A_\pi^\pi$ is uniformly distributed and we may define $\rho:=f\left(A_\pi^\pi\right)$ 
for some $f\in\mathcal{C}^1[0,1],\ f>0,$ with $\int_0^1f(x)dx=1$. Then $\E^\p\rho=1$ and we have
\[h_t=\E\left(\rho\1_{\{\sigma>t\}}|\F_t\right)=\E\left(f(A^\pi_\pi)\1_{\{\sigma>t\}}|\F_t\right)
=\E\left(f(A^\pi_\pi)\1_{\{\sigma>t\geq\pi\}}|\F_t\right)+\E\left(f(A^\pi_\pi)\1_{\{\pi>t\}}|\F_t\right).\]
The second term on the RHS has already been computed in Example \ref{pst1} as
\[\E\left(f(A^\pi_\pi)\1_{\{\pi>t\}}|\F_t\right)=\int_{A_t^\pi}^1f(x)dx.\]
Concerning the first term we have
	\[\p(\sigma>t\geq\pi|\ \F_t)=\p(\sigma>t|\F_t)-\p(\pi>t|\F_t)=Z_t^\sigma-Z_t^\pi\]
and
\begin{eqnarray*}
\E\left(f(A^\pi_\pi)\1_{\{\sigma>t\geq\pi\}}|\F_t\right)&=&\E\left(f(1-\ul{Z}^\sigma_\pi)\1_{\{\sigma>t\geq\pi\}}|\F_t\right)
=\E\left(f(1-\ul{Z}^\sigma_t)\1_{\{\sigma>t\geq\pi\}}|\F_t\right)\\
&=&f(1-\ul{Z}^\sigma_t)\cdot(Z_t^\sigma-Z_t^\pi)
=f(A^\pi_t)\cdot(Z_t^\sigma-Z_t^\pi).
\end{eqnarray*}
Hence,
\begin{eqnarray*}
h_t&=&\int_{A_t^\pi}^1f(x)dx+f(A^\pi_t)(Z_t^\sigma-Z_t^\pi)\\
dh_t&=&-f(A_t^\pi)dA_t^\pi+f'(A_t^\pi)(Z_t^\sigma-Z_t^\pi)dA_t^\pi+f(A_t^\pi)(dZ_t^\sigma-dZ_t^\pi).
\end{eqnarray*}
Since $1-A_t^\pi=Z_t^\pi=\ul{Z}_t^\sigma$, we have $supp(dA_t^\pi)\subset\{Z_t^\pi=Z_t^\sigma\}$, which implies that
\begin{eqnarray*}
dh_t=f(A_t^\pi)(dZ_t^\sigma-dZ_t^\pi-dA_t^\pi)=f(A_t^\pi)dZ_t^\sigma.
\end{eqnarray*}
Therefore, 
\begin{eqnarray*}
d\mu_t=\frac{dh_t}{h_t}=\frac{f(A_t^\pi)dZ_t^\sigma}{\int_{A_t^\pi}^1f(x)dx +f(A_t^\pi)(Z_t^\sigma-Z_t^\pi)},\qquad
d\mu_t^F=\frac{f(A_t^\pi)dA_t^\sigma}{\int_{A_t^\pi}^1f(x)dx +f(A_t^\pi)A_t^\pi},
\end{eqnarray*}
where we used that $supp(dA_t^\sigma)\subset\{Z_t^\sigma=1\}$, cf.~Lemma \ref{logN}.\\

Since $\sigma$ is honest, there exists for all $t>0$ an $\F_t$-measurable random variable $\sigma_t$ such that $\sigma=\sigma_t$ on $\{\sigma<t\}$. In fact we may choose $\sigma_t=\sup\{u\leq t:\ Z^\sigma_u=1\}$ because $\sigma=\sup\{u\geq0:\ Z^\sigma_u=1\}$ according to Lemme (5,2) of \cite{Jeulin}. Since
\[A_\pi^\pi=1-Z^\pi_\pi=1-\ul{Z}^\sigma_\pi=1-\ul{Z}^\sigma_\sigma=1-Z^\pi_\sigma=A^\pi_\sigma,\]
this implies that
\begin{eqnarray*}
k_t&=&\E^\p\left(f\left(A^\pi_\pi\right)\1_{\{\sigma\leq t\}}|\F_t\right)=\E^\p\left(f\left(A^\pi_\sigma\right)\1_{\{\sigma\leq t\}}|\F_t\right)=f\left(A^\pi_{\sigma_t}\right)(1-Z^\sigma_t)\\
dk_t&=&-f\left(A^\pi_{\sigma_t}\right)dZ^\sigma_t+\1_{\{\sigma_t=t\}}(1-Z_t^\sigma)df(A^\pi_t)=-f\left(A^\pi_{\sigma_t}\right)dZ^\sigma_t.
\end{eqnarray*}
Therefore, on $\{\sigma\leq t\}$ we have $k_t=f(A^\pi_\sigma)(1-Z_t^\sigma)$. We may now apply Theorem \ref{exhon} to conclude that for any local $(\p,\F_t)$-martingale $U$ the process 
	\[V_t=U_{t}-\int_0^{t\wedge\sigma}\frac{f(A_s^\pi)d\langle m^\sigma,U\rangle_s}{\int_{A_s^\pi}^1f(x)dx +f(A_s^\pi)(Z_s^\sigma-Z_s^\pi)}+\int_\sigma^{\sigma\vee t}\frac{d\langle m^\sigma,U\rangle_s}{1-Z^\sigma_s}\]
is a local $(\q,\G_t)$-martingale. The result is not surprising, of course, since \mbox{$\rho=A^\pi_\pi=A^\pi_\sigma\in\G_\sigma$.} Therefore the measure change has no effect after $\sigma$ and we do indeed recover the usual term from the enlargement formula under $\p$ on the interval $[\sigma\wedge t,t]$, which can for example be found in~ \cite{Jeulin}, Th\'eor\`eme (5,10). \\

Let us briefly recall Example 1 from \cite{Mortimer} to see how it fits in the above framework.

\begin{ex}\label{bpd}
For a standard Brownian motion $B$ one defines the random times
\[\sigma=\sup\{t<T_1^B:\ B_t=0\},\qquad \pi=\sup\{t<\sigma:\ B_t=\ol{B}_t\},\]
i.e.~$\sigma$ is the time of the last zero of $B$ before it first hits one, and $\pi$ is the last time at which $B$ reaches its supremum before $\sigma$. From Example \ref{b} we know that $\sigma$ is an honest time with $Z_t^\sigma=\p(\sigma>t|\F_t)=1-B_{t\wedge T_1^B}^+$ and that $\pi$ is a pseudo-stopping time constructed via Lemma \ref{inf} with $Z_t^\pi=1-\ol{B}_{t\wedge T_1^B}$. In this case $A_\pi^\pi=\ol{B}_\sigma$ and the above calculations combined with L\'evy's theorem show that the process
\[W_t:=B_t+\int_0^{t\wedge\sigma}\frac{\1_{\{B_t>0\}}f(\ol{B}_t)dt}{\int^1_{\ol{B}_t}f(y)dy+f(\ol{B}_t)(\ol{B}_t-B_t^+)}-\int_{\sigma}^{t\wedge T_1^B}\frac{dt}{B_t}\]
is a $(\q,\G_t)$-Brownian motion on the interval $\left[0,T_1^B\right]$. Especially, if we choose $\rho\equiv1$ we recover the well-known path decomposition result of the standard Brownian motion due to Williams.
\end{ex}

\s{No arbitrage up to a random time}\label{finance}

In this section we fix a filtered probability space $\left(\Omega,\F,(\F_t)_{t\geq0},\p\right)$ satisfying assumption $(C)$, on which we model a financial market consisting of a riskfree bond and a risky stock $(S_t)$, which is assumed to be a continuous $(\F_t)$-semimartingale. For simplicity, we assume that the interest rate is equal to zero and that $\F=\F_\infty:=\bigvee_{t\geq0}\F_t$.

For the reader's convenience we first repeat some notions commonly used in finance: An $a$-admissible trading strategy is any $(\F_t)$-predictable process $(\theta_t)$ which is $(S_t)$-integrable such that the value process
\[V(x,\theta)_t:=x+\int_0^t\theta_sdS_s\]
satisfies $V(0,\theta)_t\geq -a$ $\p$-almost surely for all $t\geq0$ and the limit $\lim_{t\ra\infty}V(0,\theta)_t$ exists a.s.
A process $(\theta_t)$ is called an admissible trading strategy if it is an $a$-admissible trading strategy for some $a\in\R_+$. The notion of admissibility allows us to define two different no arbitrage concepts. 

\begin{df}\label{NFLVR}
In the market model $\left(S_t,\F_t,\p\right)$ there is 
\bi
\item an Arbitrage of the First Kind if and only if there exists a non-negative $\F$-measurable random variable $\xi$ with $\p(\xi>0)>0$ such that for all $a>0$ there exists an $a$-admissible trading strategy $\theta$ such that $V(a,\theta)_\infty\geq\xi$ almost surely. If there is no arbitrage of the first kind, we say that the market satisfies the NA1 (No Arbitrage of the First Kind) condition.
\item a Free Lunch with Vanishing Risk (FLVR) if and only if there exists an $\eps>0$ and a sequence $(\theta^n)$ of $(\F_t)$-admissible strategies together with an increasing sequence $(\delta_n)$ of positive numbers converging to one such that $\p(V(0,\theta^n)_\infty>-1+\delta_n)=1$ and $\p(V(0,\theta^n)_\infty>\eps)\geq\eps$. Otherwise we say that the market satisfies the NFLVR (No Free Lunch with Vanishing Risk) condition.
\ei
\end{df}

{\bf Throughout this section we will suppose that the market model $(S_t,\F_t,\p)$  satisfies NFLVR.} A natural question is now, if the market is still arbitrage free after adding new information by enlarging the filtration progressively with a random time $\sigma$. We will again suppose that $\sigma$ avoids $(\F_t)$-stopping times, i.e.~assumption (A) is satisfied.\vskip6pt

It was shown in \cite{Kardaras} that NA1 fails, if $S$ is not a semimartingale. Moreover, according to Theorem 7.2 in \cite{NFLVR} there exists a free lunch with vanishing risk using only simple trading strategies, if $S$ is not a semimartingale. Since in general under a progressive enlargement of filtration $S$ only remains a semimartingale until time $\sigma$, we will in the following restrict ourselves to the question whether the market $\left(S_{t\wedge\sigma},\G_{t\wedge\sigma},\p\right)$ is arbitrage-free. In the case where $\sigma$ is an honest time this question has been discussed in detail by \cite{honest}. Note also that the question of the existence of an equivalent local martingale measure on the whole time horizon $[0,\infty)$ has previously been addressed in \cite{CJN}, where its connection to the so called $(\mathcal{H})$-hypothesis has been pointed out.

\subs{NFLVR on $[0,\sigma\wedge T]$}

The following theorem gives a necessary criterion to have NFLVR on the time horizon $[0,T\wedge\sigma]$, where $T$ is an $(\F_t)$-stopping time. In the case of $\sigma$ being an honest time the following statement can be found in \cite{honest} together with a long technical proof. However, we will give an apparently new proof of the statement, valid for all random times that avoid stopping times, which appeals to purely intuitive reasoning. For this we work directly with the above definition of NFLVR.

\begin{thm}\label{V}
Let $T$ be an $(\F_t)$-stopping time. If $\p\left(T_0^{Z^\p}\leq T\right)=0$, then NFLVR also holds in the enlarged financial market on the time horizon $[0,\sigma\wedge T]$.
\end{thm}

The idea of the proof is that even at time $T$ we cannot be sure that $\sigma$ has already occured because $\p\left(T_0^{Z^\p}\leq T\right)=0$.

\begin{proof}
First note that
\[ \p\left(T_0^{Z^\p}\leq T\right)=0\quad\LRA\quad Z^\p_T>0\quad \p\text{-a.s.}\] 
We proceed by contradiction: Assume that there is a FLVR in the enlarged market on the time horizon $[0,\sigma\wedge T]$. Then there exists a sequence of $(\G_t)$-admissible trading strategies $(\theta^n)_{n\in\N}$ and an increasing deterministic sequence $(\delta_n)$ converging towards $1$ such that for some $\eps>0$ and all $n\in\N$,
\[\p\left(V(0,\theta^n)_{\sigma\wedge T}>-1+\delta_n\right)=1,\quad \p\left(V(0,\theta^n)_{\sigma\wedge T}>\eps\right)\geq\eps.\]
Using Lemma \ref{pred} of the appendix we can find for every $n\in\N$ an $(\F_t)$-predictable process $(y_t^n)$ such that $\theta^n$ and $y^n$ agree almost surely up to time $\sigma$, i.e.
\[\theta^n_t\1_{\{t\leq\sigma\}}=y_t^n\1_{\{t\leq \sigma\}}\quad\forall\ t\geq0.\]
First, we will show that each $y^n$ is $S$-integrable up to time $T$. For this let us denote by $S=S_0+M+B$ the $(\F_t)$-semimartingale decomposition of $S$, where $M\in \mathcal{M}_{loc}(\F_t)$ and $B$ is of finite variation. 
By the enlargement formula (\ref{ef}) there exists a local $(\G_t)$-martingale $\tilde{M}$ such that
\[M_{t\wedge\sigma}=\tilde{M}_{t}+\int_0^{t\wedge\sigma}\frac{d\langle M,m^\p\rangle_u}{Z_u^\p},\quad t\geq0.\]
Hence, the $(\G_t)$-semimartingale decomposition of $S$ up to time $\sigma$ is given by
\[S_{t\wedge\sigma}=S_0+\tilde{M}_t+\tilde{B}_{t\wedge\sigma}\quad\text{with}\quad \tilde{B}_t:=B_t+\int_0^{t\wedge\sigma}\frac{d\langle M,m^\p\rangle_u}{Z_u^\p}.\]
Because $S,M,$ and $\tilde{M}$ are continuous and hence also $B$ and $\tilde{B}$ are continuous, the quadratic variation processes of $S,M,$ and $\tilde{M}$ are almost surely equal and do not depend on the filtration. Since $\theta^n$ is admissible, it is $S$-integrable on $[0,\sigma\wedge T]$ and we have $\int_0^{\sigma\wedge T}(\theta^n_u)^2d\langle S\rangle_u<\infty$ as well as $\int_0^{\sigma\wedge T}|\theta^n_u||d\tilde{B}_u|<\infty$ almost surely. 
Moreover by the Kunita-Watanabe inequality,
\[\int_0^{\sigma\wedge T}|\theta^n_u|\frac{|d\langle M,m^\p\rangle_u|}{Z_u^\p}\leq \left(\int_0^{\sigma\wedge T}(\theta^n_u)^2d\langle S\rangle_u\right)^{1/2} \left(\int_0^{\sigma\wedge T}\frac{d\langle m^\p\rangle_u}{Z_u^\p}\right)^{1/2}<\infty \quad\text{a.s.},\]
because $m^\p$ is a uniformly integrable martingale and $T_0^{Z^\p}>\sigma$ almost surely, cf.~Lemme (4,3) of \cite{Jeulin}. Therefore, also $\int_0^{\sigma\wedge T}|\theta_u^n||dB_u|<\infty$ almost surely. Now observe that 
\begin{eqnarray*}
0&=&\p\left(\int_0^{\sigma\wedge T}|\theta^n_u||dB_u|=\infty\right)=
\p\left(\int_0^{\sigma\wedge T}|y^n_u||dB_u|=\infty\right)\\
&\geq&
\p\left(\sigma>T;\ \int_0^{T}|y^n_u||dB_u|=\infty\right)=
\E^\p\left(Z^\p_T\1_{\left\{\int_0^T|y^n_u||dB_u|=\infty\right\}}\right)
\end{eqnarray*}
and similarly
\begin{eqnarray*}
0&=&\p\left(\int_0^{\sigma\wedge T}\left(\theta^n_u\right)^2d\langle S\rangle_u=\infty\right)=
\p\left(\int_0^{\sigma\wedge T}\left(y^n_u\right)^2d\langle S\rangle_u=\infty\right)\\
&\geq&
\p\left(\sigma>T;\ \int_0^T\left(y^n_u\right)^2d\langle S\rangle_u=\infty\right)=
\E^\p\left(Z^\p_T\1_{\left\{\int_0^T(y^n_u)^2d\langle S\rangle_u=\infty\right\}}\right).
\end{eqnarray*}
Since $Z^\p_T>0$ a.s.~we conclude that $\int_0^T|y^n_u||dB_u|<\infty$ and $\int_0^T(y^n_u)^2d\langle S\rangle_u<\infty$ almost surely, i.e.~$y^n$ is $S$-integrable up to time $T$. 

Furthermore, 
since $\theta^n$ is admissible there exists $a_n\in\R_+$ such that for all $t\geq0$,
\[\p(V(0,\theta^n)_{\sigma\wedge t\wedge T}>-a_n)=1.\]
We will prove that also
\[\p(V(0,y^n)_{t\wedge T}>-a_n)=1\quad\forall\ t\geq0.\]
Assume that this was not the case, i.e.~there exists $t\geq0$ with
\[\p(V(0,y^n)_{t\wedge T}\leq -a_n)>0.\]
Since $Z^\p_T>0$ almost surely, this would imply that
\begin{eqnarray*}
0<\E^\p\left(\1_{\{V(0,y^n)_{t\wedge T}\leq -a_n\}}Z^\p_T\right)&=&\p(V(0,y^n)_{T\wedge t}\leq -a_n;\ \sigma>T)\\
&=&\p(V(0,\theta^n)_{\sigma\wedge t\wedge T}\leq -a_n;\ \sigma>T)\\
&\leq&\p(V(0,\theta^n)_{\sigma\wedge t\wedge T}\leq -a_n)=0,
\end{eqnarray*}
a contradiction. Thus, each $y^n$ is an admissible strategy for the $(S_t,\F_t,\p)$ market. Moreover, by the same reasoning as above one can show that 
\[\p(V(0,y^n)_T>-1+\delta_n)=1.\]
For every $n\in\N$ we define the $(\F_t)$-trading strategy
\[\vartheta^n_t:=y_t^n\1_{\{0\leq t\leq T^n_\eps\}},\]
where
\[T^n_\eps:=\inf\{t\geq0:\ V(0,y^n)_t=\eps\}.\]
Clearly, $\vartheta^n$ is admissible as well and
\[\p(V(0,\vartheta^n)_T>-1+\delta_n)\geq \p(V(0,y^n)_T>-1+\delta_n)=1.\]
Moreover, 
\begin{eqnarray*}
\p\left(V(0,\vartheta^n)_T>\frac{\eps}{2}\right)&\geq&
\p\left(T^n_\eps\leq T\right)=\p\left(\exists\ u\leq T:\ V(0,y^n)_u\geq\eps\right)\\
&\geq&\p\left(\exists\ u\leq\sigma\wedge T:\ V(0,y^n)_u\geq\eps\right)\\
&=&\p\left(\exists\ u\leq\sigma\wedge T:\ V(0,\theta^n)_u\geq\eps\right)\\
&\geq&\p\left(V(0,\theta^n)_{\sigma\wedge T}\geq\eps\right)>\eps.
\end{eqnarray*}
Choosing $\tilde{\eps}:=\eps/2$, this would give a FLVR with respect to $(\F_t)$, which is impossible by assumption.
\end{proof}

\begin{rem}
In fact it is sufficient to require in the statement of Theorem \ref{V} that $\q\left(T_0^{Z^{\q}}\leq T\right)$ for some measure $\q\sim\p$. This follows directly from the statement of the theorem, but can also be seen as follows: since $\rho:=\frac{d\q}{d\p}>0$, we have
\[\{Z^\p>0\}=\{h>0\}=\{Z^{\q}>0\}\quad\RA\quad T_0^{Z^\p}=T_0^{Z^{\q}}\ \text{a.s.}\]
\end{rem}

\subs{Local martingale deflators and equivalent local martingale measures on $[0,\sigma\wedge T]$}

Instead of working directly with the definition of NFLVR, one can also make use of the fundamental theorem of asset pricing, cf.~Theorem \ref{na}, and look for the existence of the dual variables defined below. This approach will be used in what follows.

\begin{df}\label{dual}
In the market model $\left(S_t,\F_t,\p\right)$ we call
\bi
\item a strictly positive local $(\F_t)$-martingale $(L_t)$ with $L_0=1$ and $L_\infty>0$ a.s.~a local martingale deflator, if the process $(L_tS_t)$ is a local $(\F_t)$-martingale.
\item $\widetilde{\p}:=L_\infty.\p$ an equivalent local martingale measure (ELMM), if there exists a local martingale deflator $(L_t)$ which is a uniformly integrable \mbox{martingale closed by $L_\infty$.}
\ei
\end{df}

The proof of the following very important theorem can be found in \cite{NFLVR} and \cite{Kad}, noting that the proof in \cite{Kad} carries over to the infinite time horizon case. 

\begin{thm}\label{na}
In the financial market model $\left(S_t,\F_t,\p\right)$
\bi
\item the NA1 condition is equivalent to the existence of a local martingale deflator.
\item the NFLVR condition is equivalent to the existence of an ELMM. 
\ei
\end{thm}

In the following we will approach the question of NA1 / NFLVR up to a random time by looking for local martingale deflators / ELMMs in the enlarged filtration. Throughout we will denote by $\q=\rho.\p$ an ELMM for the $(S_t,\F_t,\p)$ market which exists due to Theorem \ref{na} because the market is assumed to satisfy NFLVR. As before we denote by $\rho_t:=\E^\p(\rho|\F_t),\ t\geq0$, its Radon-Nikoydm derivative with respect to $(\F_t)$. Moreover, we denote by $Z^\p=N^\p D^\p$ the It\^o-Watanabe decomposition of $Z^\p$, cf.~Remark \ref{nd}.\\

The following Lemma was proven in \cite{honest} in the case of honest times with $\q=\p$, where it was remarked that it also holds in greater generality. For completeness we provide a proof as well.

\begin{lem}\label{deflate}
The process $(\rho_{t\wedge\sigma}/N^{\p}_{t\wedge\sigma})_{t\geq0}$ is a local martingale deflator for $(S_{t\wedge\sigma})$ in the filtration $(\G_t)$, i.e.~$NA1$ holds with respect to $(\G_t)$ on the time horizon $[0,\sigma]$.
\end{lem}

\begin{proof}
First note that the process $(\rho_{t\wedge\sigma}/N^{\p}_{t\wedge\sigma})$ is well-defined, since $T_0^{Z^\p}>\sigma$ a.s. If $X\in\mathcal{M}_{loc}(\p,\F_t)$, then by the enlargement formula (\ref{ef2}) the processes 
\[\tilde{X}_t:=X_{t\wedge\sigma}-\int_0^{t\wedge\sigma}\frac{d\langle X,N^\p\rangle_s}{N^\p_s}\]
and
\[\tilde{N}^\p_t:=N^\p_{t\wedge\sigma}-\int_0^{t\wedge\sigma}\frac{d\langle N^\p\rangle_s}{N^\p_s}\]
are local $(\p,\G_t)$-martingales. With It\^o's formula we therefore have on $[0,\sigma]$,
\begin{eqnarray*}
d\left(\frac{X}{N^\p}\right)&=&\frac{dX}{N^\p}-\frac{X}{(N^\p)^2}dN^\p+\frac{X}{(N^\p)^3}d\langle N^\p\rangle-\frac{d\langle X,N^\p\rangle}{(N^\p)^2}\\
&=&\frac{X}{N^\p}\left(\frac{d\tilde{X}}{X}+\frac{d\langle X,N^\p\rangle}{XN^\p}-\frac{d\tilde{N}^\p}{N^\p}-\frac{d\langle N^\p\rangle}{(N^\p)^2}+\frac{d\langle N^\p\rangle}{(N^\p)^2}-\frac{d\langle X,N^\p\rangle}{XN^\p}\right)\\
&=&\frac{X}{N^\p}\left(\frac{d\tilde{X}}{X}-\frac{d\tilde{N}^\p}{N^\p}\right).
\end{eqnarray*}
Hence, $(X_{t\wedge\sigma}/N^\p_{t\wedge\sigma})\in\mathcal{M}_{loc}(\p,\G_t)$. Especially, taking $X=(\rho_t)$ yields that $(\rho_{t\wedge\sigma}/N^\p_{t\wedge\sigma})\in\mathcal{M}_{loc}(\p,\G_t)$. Since $N^\p$ is a non-negative local $(\p,\F_t)$-martingale, it does not explode. Therefore, $\p(N^\p_\sigma=\infty)=0$ and $\frac{\rho_\sigma}{N^\p_\sigma}>0$ a.s. Moreover, we may choose $X=(\rho_tS_t)$, which is local $(\p,\F_t)$-martingale because $(\rho_t)$ is a local martingale deflator in the $(S_t,\F_t,\p)$ market. This yields that 
\[\left(\frac{S_{t\wedge\sigma}\rho_{t\wedge\sigma}}{N^\p_{t\wedge\sigma}}\right)\in\mathcal{M}_{loc}(\p,\G_t)\]
and therefore $(\rho_{t\wedge\sigma}/N^{\p}_{t\wedge\sigma})$ is a local martingale deflator in the enlarged filtration on $[0,\sigma]$.
\end{proof}

\begin{rem}
The validity of NA1 in a progressively enlarged filtration has recently been proven to hold in much greater generality without assuming (AC), cf.~\cite{Acciaio,Anna1,Anna2}.
\end{rem}

In \cite{honest} it is shown that for honest times the condition $\p\left(T^{Z^\p}_0\leq T\right)=0$, which we derived in Theorem \ref{V}, is not only sufficient but also necessary to have NFLVR on $[0,T\wedge\sigma]$ in a complete market. However, the condition $\p\left(T^{Z^\p}_0\leq T\right)=0$ is not in general necessary, even in a complete market, as the following example shows.

\begin{ex}
Let $\sigma$ be a $\p$-pseudo-stopping time bounded by one. Then $1-Z^\p_1=A^\p_1=1$ and therefore $\p\left(T_0^{Z^\p}\leq 1\right)=1$. However, since $\sigma$ is a $\p$-pseudo-stopping time, $\E^\p\rho_\sigma=1$ and $N^\p\equiv1$. Therefore, $(\rho_{t\wedge\sigma})$ is a local martingale deflator with $\E^\p\rho_\sigma=1$ due to Lemma \ref{deflate} and thus a uniformly integrable martingale, which defines an ELMM in the $(S_{t\wedge\sigma},\G_{t\wedge\sigma},\p)$ market model. Hence, NFLVR holds in the enlarged market on the interval $[0,\sigma]=[0,\sigma\wedge 1]$.
\end{ex}

Next we prove a sufficient and necessary criterion such that $(\rho_{t\wedge\sigma}/N^\p_{t\wedge\sigma})$ is a uniformly integrable martingale on the time interval $[0,\sigma\wedge T]$, where $T$ is an $(\F_t)$-stopping time. As it turns out this criterion will be less restrictive than the condition $\p\left(T^{Z^\p}_0\leq T\right)=0$ derived in Theorem \ref{V}.

\begin{thm}\label{nfl}
Let $T$ be an $(\F_t)$-stopping time. Then,
\[\left(\frac{\rho_{t\wedge\sigma\wedge T}}{N^\p_{t\wedge\sigma\wedge T}}\right)\in\mathcal{M}_{u.i.}(\p,\G_t)\quad\LRA\quad\E^\p\left(D^\p_\infty\1_{\{T_0^{N^\p}\leq T\}}\right)=0.\]
\end{thm}

\begin{proof}
The local $(\G_t)$-martingale $(\rho_{t\wedge\sigma\wedge T}/N^\p_{t\wedge\sigma\wedge T})_{t\geq0}$ is a uniformly integrable martingale if and only if $\E^\p(\rho_{\sigma\wedge T}/N^\p_{\sigma\wedge T})=1$. Since $\sigma<T_0^{Z^\p}=T_0^{N^\p}\wedge T_0^{D^\p}$ almost surely, 
\begin{eqnarray*}
\E^\p\left(\frac{\rho_{\sigma\wedge T}}{N^\p_{\sigma\wedge T}}\right)&=&\E^\p\left(\1_{\left\{T_0^{Z^\p}>\sigma\right\}}\frac{\rho_{\sigma\wedge T}}{N^\p_{\sigma\wedge T}}\right)
=\E^\p\left(\int_0^{T_0^{Z^\p}}\frac{\rho_{u\wedge T}}{N^\p_{u\wedge T}}dA^\p_u\right)\\
&=&\E^\p\left(\int_0^{T_0^{Z^\p}\wedge T}\frac{\rho_u}{N^\p_u}dA^\p_u+\1_{\left\{T_0^{Z^\p}>T\right\}}\frac{\rho_T}{N^\p_T}\left(A^\p_{T_0^{Z^\p}}-A^\p_T\right)\right)\\
&=&\E^\p\left(-\int_0^{T_0^{Z^\p}\wedge T}\rho_udD^\p_u+\1_{\left\{T_0^{Z^\p}>T\right\}}\frac{\rho_T}{N^\p_T}\left(m^\p_{T_0^{Z^\p}}-A^\p_T\right)\right)\\
&=&\E^\q\left(1-D^\p_{T_0^{Z^\p}\wedge T}+\1_{\left\{T_0^{Z^\p}>T\right\}}\frac{Z^\p_T}{N_T^\p}\right)\\
&=&\E^\q\left(1-D^\p_{T_0^{Z^\p}\wedge T}+\1_{\left\{T_0^{Z^\p}>T\right\}}D^\p_T\right)=1-\E^\q\left(D^\p_{T_0^{Z^\p}}\1_{\left\{T_0^{Z^\p}\leq T\right\}}\right)\\
&=&1-\E^\q\left(D^\p_{T_0^{Z^\p}}\1_{\left\{T_0^{N^\p}\leq T\right\}}\right)=1-\E^\q\left(D^\p_\infty\1_{\left\{T_0^{N^\p}\leq T\right\}}\right),
\end{eqnarray*}
where in the last equality we used that $supp(dD^\p)\subset\{Z^\p>0\}$, cf.~Remark \ref{nd}. Finally note that
\[\E^\q\left(D^\p_\infty\1_{\{T_0^{N^\p}\leq T\}}\right)=0\ \LRA\ \E^\p\left(D^\p_\infty\1_{\{T_0^{N^\p}\leq T\}}\right)=0.\]
\end{proof}

\begin{rem}\label{honestn}
For an honest time $\sigma$ the multiplicative decomposition of $Z^\p$ is given by $Z^\p_t=N^\p_t/\ol{N}^\p_t$, where $N^\p$ is a non-negative local martingale converging to zero almost surely, cf.~Lemma \ref{Max}. And since a non-negative local martingale does not explode almost surely,
\[D^\p_\infty=\frac{1}{\ol{N}^\p_\infty}>0\quad\text{a.s.}\]
Therefore, the process $\left(\rho_{t\wedge\sigma\wedge T}/N^\p_{t\wedge\sigma\wedge T}\right)$ is a uniformly integrable martingale if and only if $\p\left(T_0^{N^\p}\leq T\right)=\p\left(T_0^{Z^\p}\leq T\right)=0$. Especially, if $T_0^{N^\p}=\infty$ almost surely, $(\rho_{\sigma\wedge t}/N^\p_{\sigma\wedge t})_{t\geq0}$ is actually a true martingale and not a strict local martingale, cf.~also Remark 3.6 in \cite{honest}. Note however that Theorem \ref{nfl} implies that $\left(\rho_{t\wedge\sigma}/N^\p_{t\wedge\sigma}\right)$ is \textit{never} a uniformly integrable martingale.
\end{rem}

We can now derive the result of Theorem \ref{V} as a Corollary of Theorem \ref{nfl}.

\begin{cor}\label{cord}
Let $T$ be an $(\F_t)$-stopping time. If $\p\left(T_0^{Z^\p}\leq T\right)=0$, then NFLVR holds in the enlarged market on the time interval $[0,T\wedge\sigma]$.
\end{cor}

\begin{proof}
If $\p\left(T_0^{Z^\p}\leq T\right)=0$, then 
\[\p\left(T\geq T_0^{N^\p}\right)=\p\left(T\geq T_0^{N^\p}\geq T_0^{Z^\p}\right)=0.\]
Hence, the claim follows from Theorem \ref{nfl}, Lemma \ref{deflate}, and Theorem \ref{na}.
\end{proof}

Moreover, taking $T=\infty$ in Theorem \ref{nfl} we get the following corollary.

\begin{cor}\label{dinfty}
If $D^\p_\infty=0$ almost surely, then NFLVR holds on the interval $[0,\sigma]$ with respect to the filtration $(\G_t)$.
\end{cor}

Of course, every pseudo-stopping time fulfills $D^\p_\infty=1-A^\p_\infty=1-1=0$. The following example, known as \'Emery's example, shows that there are also other random times which fulfill the assumption of Corollary \ref{dinfty} and thus allow for an equivalent local martingale measure up to time $\sigma$. 

\begin{ex}
Let $W$ be a $(\p,\F_t)$-Brownian motion and set $\sigma=\sup\{t\leq1:\ 2W_t=W_1\}$. The corresponding Az\'ema supermartingale is
\[Z^\p_t=\sqrt{\frac{2}{\pi}}\int_{\frac{|W_t|}{\sqrt{1-t}}}^\infty x^2e^{-x^2/2}dx=
m^\p_t-\sqrt{\frac{2}{\pi}}\int_0^t\frac{|W_u|}{(1-u)^{3/2}}\exp\left(-\frac{W_u^2}{2(1-u)}\right)du\]
with $m^\p\not\equiv1$, cf.~section 5.6.5 in \cite{MMM}. 
For every $n\in\N$ define the set 
\[B_n=\left\{|W_u|>\sqrt{\frac{2}{n}}\ \forall\ u\in\left[1-\frac{1}{n},1\right]\right\}\] 
and note that
\[1=\p(W_1\neq 0)=\lim_{n\ra\infty}\p(B_n).\]
On the set $B_n$ we have for all $u\in\left[1-\frac{1}{n},1\right]$,
\[\frac{|W_u|}{\sqrt{1-u}}>\sqrt{2}\]
and hence
\[\frac{1}{2}\int_{\frac{|W_u|}{\sqrt{1-u}}}^\infty x^2e^{-x^2/2}dx\leq\int_{\frac{|W_u|}{\sqrt{1-u}}}^\infty(x^2-1)e^{-x^2/2}dx=\frac{|W_u|}{\sqrt{1-u}}\exp\left(-\frac{W_u^2}{2(1-u)}\right).\]
Thus, the following estimate holds on $B_n$:
\begin{eqnarray*}
\int_0^1\frac{dA^\p_t}{Z^\p_t}&\geq&\int_{1-\frac{1}{n}}^1\frac{dA^\p_t}{Z^\p_t}
=\int_{1-\frac{1}{n}}^1\frac{dA^\p_t}{\sqrt{\frac{2}{\pi}}\int_{\frac{|W_t|}{\sqrt{1-t}}}^\infty x^2e^{-x^2/2}dx}\\
&\geq&\frac{1}{2}\int_{1-\frac{1}{n}}^1\frac{dA^\p_t}{\sqrt{\frac{2}{\pi}}\frac{|W_t|}{\sqrt{1-t}}\exp\left(-\frac{W_t^2}{2(1-t)}\right)}
=\frac{1}{2}\int_{1-\frac{1}{n}}^1\frac{dt}{1-t}=\infty.
\end{eqnarray*}
Therefore, on each $B_n$ we have
\[D^\p_\infty=D^\p_1=\exp\left(-\int^1_0\frac{dA^\p_t}{Z^\p_t}\right)=0,\]  
and by monotone convergence 
\[\E^\p\left(D^\p_\infty\right)=\lim_{n\ra\infty}\E^\p\left(D^\p_\infty\1_{B_n}\right)=0\quad\LRA\quad D^\p_\infty=0\quad \p\text{-a.s.}\]
\end{ex}

\appendix

\s{}

\subs{A useful lemma}

The following well-known Lemma can for example be found in paragraph XX.75 of \cite{DM5}.

\begin{lem}\label{pred}
\mbox{}
\be
\item If $G$ is a $(\G_t)$-predictable process, then there exists an $(\F_t)$-predictable process $F$ such that 
\[G_t\1_{\{t\leq\sigma\}}=F_t\1_{\{t\leq\sigma\}},\quad t\geq0.\]
\item If $\xi$ is a $\p$-integrable variable, then
\[\E^\p(\xi\1_{\{\sigma>t\}}|\G_t)=\1_{\{\sigma>t\}}\frac{\E^\p(\xi\1_{\{\sigma>t\}}|\F_t)}{Z_t^\p}.\]
\item If $T$ is a $(\G_t)$-stopping time, then there exists an $(\F_t)$-stopping time $S$ such that
\[T\wedge\sigma=S\wedge\sigma.\]
\ee
\end{lem}

\subs{Condition $(P)$}\label{AppP}

The following condition goes back to Parthasarathy, cf.~\cite{para}, and was labeled condition $(P)$ in \cite{newkind}.

\begin{df}
Let $(\Omega, \F, (\F_t )_{t\geq0} )$ be a filtered measurable space, such that $\F$  is the $\sigma$-
algebra generated by $(\F_t)_{t\geq0}$:  $\F = \bigvee_{t\geq0}\F_t$.
We shall say that the property $(P)$ holds if
and only if $(\F_t )_{t\geq0}$ enjoys the following conditions:
\bi
\item For all $t \geq 0$, $\F_t$ is generated by a countable number of sets.
\item For all $t \geq 0$, there exists a Polish space $\Omega_t$, and a surjective map $\pi_t$ from $\Omega$ to $\Omega_t$,
such that $\F_t$ is the $\sigma$-algebra of the inverse images by $\pi_t$ of Borel sets in $\Omega_t$, and
such that for all $B\in\F_t$ , $\omega\in\Omega$, $\pi_t(\omega)\in\pi_t (B)$ implies $\omega\in B$.
\item If $(\omega_n )_{ n\geq0}$ is a sequence of elements of $\Omega$ such that for all $N\geq0$,
	\[\bigcap_{n\geq0}^N A_n (\omega_n ) \neq\emptyset,
\]
where $A_n (\omega_n )$ is the intersection of the sets in $\F_n$ containing $\omega_n$, then
	\[\bigcap_{n\geq0}^\infty A_n (\omega_n ) \neq\emptyset.\]
\ei
\end{df}

\end{document}